\documentclass[12pt,reqno]{amsart}
\usepackage[margin=1.5cm]{geometry}
\usepackage{graphicx} % Required for inserting images
\usepackage{url}
\usepackage[utf8]{inputenc}
\usepackage{geometry,amsmath,amsthm,amssymb}
\usepackage[usenames]{color}
\usepackage{amsmath,amsthm,pdfsync,verbatim,graphicx,epstopdf,enumerate}
\usepackage{mathtools}
\usepackage[colorlinks=true]{hyperref}
\usepackage{cancel}
\usepackage[framemethod=tikz]{mdframed}
\hypersetup{allcolors=blue}
\numberwithin{equation}{section}
\newcommand{\I}{\mathrm{i}}

\newcommand{\lb}{\left(}

\newcommand{\rb}{\right)}
\newcommand{\PD}{\partial}

\newcommand{\Beq}{\begin{equation}}
	\newcommand{\Eeq}{\end{equation}}
\newcommand{\beq}{\begin{equation*}}
	\newcommand{\eeq}{\end{equation*}}
\newcommand{\bal}{\begin{align}}
	\newcommand{\eal}{\end{align}}
\renewcommand{\O}{\Omega}

\newcommand{\g}{\gamma}

\newcommand{\n}{\nabla}

\usepackage{mathtools}

\newcommand{\B}{\beta}
\newcommand{\bp}{\begin{prob}}
	\newcommand{\ep}{\end{prob}}
\newcommand{\bpr}{\begin{proof}}
	\newcommand{\epr}{\end{proof}}

\newcommand{\tred}[1]{#1}

\newcommand{\bel}[1]{\begin{equation}\label{#1}}
	\newcommand{\ee}{\end{equation}}

\newtheorem{theorem}{Theorem}[section]
\newtheorem{corollary}[theorem]{Corollary}

\newtheorem{lemma}[theorem]{Lemma}

\theoremstyle{definition}
\newtheorem{definition}[theorem]{Definition}

\newtheorem{remark}[theorem]{Remark}

\newcommand{\D}{\mathrm{d}}

\newcommand{\Rb}{\mathbb{R}}
\newcommand{\Dc}{\mathcal{D}}
\newcommand{\Nc}{\mathcal{N}}
\newcommand{\A}{\alpha}
\newcommand{\vp}{\varphi}

\usepackage[usenames]{color}
\usepackage{amsmath,pdfsync,verbatim,graphicx,epstopdf,enumerate}

\newcommand{\wt}{\widetilde}

\newcommand{\Cc}{\mathcal{C}}

\renewcommand{\O}{\Omega}

\newcommand{\seanc}[1]{\begin{quotation}\textbf{Sean's comment:\
}{\textit{#1}}\end{quotation}}

\title[Ray transform on Riemannian manifolds with conjugate points]{Ray transform of symmetric tensor fields on Riemannian manifolds with conjugate points}
\author[Holman and Krishnan]{Sean Holman$^*$ and Venkateswaran P.\ Krishnan$^\dagger$}
\address{$^*$Department of Mathematics, The University of Manchester,
Alan Turing Building, Oxford Rd, Manchester, M13 9PL, UK. {\tt Email:sean.holman@manchester.ac.uk}
}
\address{$^\dagger$TIFR Centre for Applicable Mathematics, PO Box 6503, GKVK PO, Sharada Nagar, Chikkabommasandra, Yelahanka New Town, Bangalore, Karnataka 560065, India. {\tt Email:vkrishnan@tifrbng.res.in}
}
\begin{document}

\begin{abstract}
    In this article, we study the microlocal properties of the geodesic ray transform of symmetric $m$-tensor fields on 2-dimensional Riemannian manifolds with boundary allowing the possibility of conjugate points. As is known from an earlier work on the geodesic ray transform of functions in the presence of conjugate points, the normal operator can be decomposed into a sum of a pseudodifferential operator ($\Psi$DO) and a finite number of Fourier integral operators (FIOs) under the assumption of no singular conjugate pairs along geodesics, which always holds in 2-dimensions. In this work, we use the method of stationary phase to explicitly compute the principal symbol of the $\Psi$DO and each of the FIO components of the normal operator acting on symmetric $m$-tensor fields. Next, we construct a parametrix recovering the solenoidal component of the tensor fields modulo FIOs, and prove a cancellation of singularities result, similar to an earlier result of Monard, Stefanov and Uhlmann for the case of geodesic ray transform of functions in 2-dimensions. We point out that this type of cancellation result is only possible in the 2-dimensional case.
\end{abstract}
\maketitle
\section{Introduction} 
Let $(M,g)$ be an $n$-dimensional Riemannian manifold with boundary. A problem of significant interest in seismology and geometric inverse problems is the unique determination of the metric $g$ from the knowledge of the boundary distance function $d_g$ on $\PD M\times \PD M$.  A region in the interior of the manifold with ``high metric'' will not be seen by the boundary distance function since it takes into account only the shortest paths connecting boundary points. Due to this reason, unique determination of $g$ from $d_g|_{\partial M \times \partial M}$ in general is not possible, and some additional restrictions on the metric $g$ are required.  One assumption is to restrict the metric under consideration to be \emph{simple}. By this we mean the boundary $\PD M$ is strictly convex and for any $x\in M$, $\exp_x: \exp_x^{-1}M\to M$ is a differomorphism.  A model $n$-dimensional simple manifold with boundary is a strictly convex domain $\O\subset \Rb^n$ with Euclidean metric. Another obstruction to uniqueness always present regardless of the restrictions on the metric $g$ results from the fact that any isometry of the manifold $M$ fixing the boundary will give the same boundary distance function. Due to this, one can determine the Riemannian metric (if at all possible) given the boundary distance function only up to diffeomorphisms fixing the boundary, in which case, we call such Riemannian manifolds boundary distance rigid. Michel conjectured that any simple Riemannian manifold is boundary distance rigid.  While this conjecture in its full generality remains open, significant progress has been made in the last few decades; see \cite{Mukhometov, Michel,Pestov-Uhlmann-Annals, SU-JAMS, SUV-Annals} for a few groundbreaking works in this direction. 
Linearization of the boundary distance function near a fixed simple metric \cite{sharafutdinov2012integral} or linearization of the lens data (which consists of the length of geodesics together with the scattering relation) in the general case \cite{SU-lens} %\shc{Does the metric need to be simple? I suppose for non-simple cases we need to consider the scattering relation.}
%\vkc{I am thinking of leaving it as it is. We state in the para below that we study the problem by relaxing the assumption that the metric be simple.}
%\shc{To do: look at references for linearisation about non-simple metrics.}
%\vkc{Add SU work on lens rigidity work to deal with the non-simple case.}
leads to the integral geometry problem of determining a symmetric $2$-tensor field $f$ from the knowledge its geodesic ray transform $I_2f$. The geodesic ray transform associates to $f$ its integrals along geodesics connecting boundary points: 
\[
I_2f(\g)=\int\limits_0^{l(\g)} f_{ij}(\g(t)) \dot{\g}^{i}(t) \dot{\g}^{j}(t) \D t,
\]
where $\g:[0,l(\g)]\to M$ is a geodesic connecting boundary points with $l(\g)$ denoting the length of the geodesic. As is expected, this transform has an infinite dimensional kernel. Any symmetric $2$-tensor field of the form $\D v$ with $v|_{\PD M}=0$ is in the kernel. Here $\D$ denotes the symmetrized covariant derivative; see \eqref{CovariantDerivative} below. This non-injectivity should be expected due to the natural obstruction to uniqueness in the non-linear problem caused by isometries fixing the boundary mentioned above.  The natural conjecture is that on a simple Riemannian manifold $(M,g)$,  $If \equiv 0$ implies that $f=\D v$ for some $v$ with $v|_{\PD M}=0$; we say $f$ is potential in this case. 

 In the above problem, one can relax the assumption that the rank of the tensor field be $2$ and this leads to the following question. Let $f$ be a symmetric $m$-tensor field, written in local coordinates as 
 \[
 f=f_{i_1 \cdots i_m} \D x^{i_1}\cdots \D x^{i_m}.
 \]
 Define the ray transform of $f$ as 
 \[
I_mf (\g)=\int\limits_0^{l(\g)} f_{i_1 \cdots i_m}(\g(t)) \dot{\g}^{i_1}(t) \cdots \dot{\g}^{i_m}(t) \D t,
\]
where $\g:[0,l(\g)]\to M$ is a geodesic connecting boundary points. %\shc{In a lot of references this is $I_m$ to indicate the degree of the tensor. Do you think we should change it?} 
%\vkc{Yes, I agree with you. But since $m$ does not play a crucial role in our analysis, I left it as it is. I will change it to $I_m$} 
The natural conjecture is that $I_mf\equiv 0$ implies $f$ is potential, that is $f=\D v$ for some symmetric $(m-1)$-tensor field $v$ vanishing on the boundary. If this is true, we say that the ray transform is $s$-injective. This conjecture has been proven under curvature assumptions that include the Euclidean case \cite{Pestov-Sharafutdinov,sharafutdinov2012integral}, for the case of simple metrics when $m=0,1$ \cite{Mukhometov, Anikonov-Romanov}, and for the case of simple metrics in two dimensions for $m=2$ in \cite{SharafutdinovJGA} and any $m$ \cite{paternain2013tensor}. For tensor fields of rank $m\geq 2$ in dimension $n\geq 3$, this problem in its full generality remains open even for simple Riemannian manifolds. Nevertheless, important progress has been made in the last few decades; see \cite{Pestov-Sharafutdinov, SU-RayT-Stability, SU-JAMS, dairbekov2006integral, Uhlmann-Vasy,SUV-Localtensors, SUV-Annals,sharafutdinov1997integral,sharafutdinov2002problem,stefanov2008integral,guillarmou2017lens}. For example, in \cite{Pestov-Sharafutdinov}, $s$-injectivity was proven under a curvature assumption based on Mukhometov's energy estimates, now formalized as Pestov identities. In \cite{SU-RayT-Stability, SU-JAMS}, the transform $I_2$ acting on symmetric 2-tensor fields was studied from a microlocal analysis point of view. In particular, it was shown that on simple Riemannian manifolds, the transform $I_2$ is a pseudodifferential operator ($\Psi$DO) of order $-1$ elliptic on solenoidal tensors and a parametrix was constructed recovering the solenoidal component of the tensor field modulo smoothing terms. For the case of real-analytic simple metrics, using tools from analytic microlocal analysis, the integral geometry problem was settled (without any assumptions on curvature) in \cite{SU-JAMS}. A significant breakthrough result was achieved by Uhlmann and Vasy in \cite{Uhlmann-Vasy} where they proved the invertibility of the local geodesic ray transform of functions in manifolds of dimensions at least 3  near the boundary under a convexity assumption on the boundary using tools from Melrose scattering  $\Psi$DO calculus. Furthermore, under the assumption of a global foliation by
strictly convex hypersurfaces, the ray transform on functions was shown to be globally injective. These ideas were later extended to the study of ray transform of symmetric $2$-tensor fields in \cite{SUV-Localtensors} and used to solve the corresponding non-linear boundary rigidity problem, under certain convexity assumptions in \cite{SUV-Annals}. We also mention a remarkable work of Guillarmou, who showed that negatively curved manifolds with strictly convex boundaries, and more generally manifolds with hyperbolic trapped sets and no conjugate points are deformation lens rigid, in particular, boundary distance rigid. For a list of open problems in the area, see \cite{paternain2014tensor}, and for a detailed contemporary reference with related results, see \cite{paternain2023geometric}. %,pestov2004characterization}. \shc{I've added some references here and this is now a long list. Maybe we should break it up and comment more specifically on the contributions?} 
%\vkc{My lazy response is to leave it as it is. But since these are fundamental contributions, we should devote some space to explain what the individual results are. I will take a shot at it and you can then modify it if needed.} 
 
%\vkc{Modified the sentence about Guillarmou's work. See if it's ok.}

One approach to understanding the geodesic ray transform, is through the study of the normal operator $\mathcal{N}_m = I^{*}_mI_m$, where $I^{*}_m$ is the formal $L^2$ adjoint of $I_m$. This has been the main object of study in several of the works mentioned in the previous paragraph. 
%While this conjecture in its full generality still remains open, as with the nonlinear boundary rigidity problem, significant progress has been made in context of simple Riemannian metrics; see \cite{Mukhometov, Anikonov-Romanov,  for a few of the impactful works in this direction. 
%determining the metric $g$ from the knoweledge of the boundary distance function $d_{g}$.  Our main object of study is the geodesic ray transform, which associates to a function and more generally a symmetric $m$-tensor field, its integrals along geodesics.
%The study of this transform began with the work of Mukhometov 
%Let $(M,g)$ be a Riemannian manifold with boundary. The geodesic ray transform $I$ associates to a symmetric $m$-tensor field, integrals along geodesics. Let $f$ is a symmetric $m$-tensor field written in local coordinates as 
%\[
%f(x)=f_{i_1 \cdots i_m}(x) \D x^{i_1}\cdots \D x^{i_m}, 
%\]
%then 
%\[
%If(\g)=\int f_{i_1 \cdots i_m}(\g(t))\dot{\g}^{i_1}(t) \cdots \dot{\g}^{i_m}(t) \D t.
%\]
When there are no conjugate points, more precisely, when the manifold $(M,g)$ is simple, it is known that the normal operator is a pseudodifferential operator ($\Psi$DO) of order $-1$, elliptic on the solenoidal component $^s\!f$ of the symmetric $m$-tensor field $f$ (see \cite{stefanov2004stability} for the case $m=2$ and \cite{SSU} for any $m\geq 1$). 
Consequently, the wavefront set $\mbox{WF}(^s\!f)$ of the solenoidal component of $f$ can be recovered from the knowledge of the geodesic ray transform. 

In the presence of conjugate points, these assertions are no longer true. 
The normal operator is no longer a $\Psi$DO. In \cite{SU-folds}, the authors showed that for the case $m=0$, in the presence of only fold-type conjugate points, the normal operator acting on functions is a sum of a pseudodifferential operator ($\Psi$DO) and a Fourier integral operator (FIO), and they computed the principal symbol of both these operators. They also showed a cancellation of singularities result for the geodesic ray transform in 2-dimensions.  
In a follow-up work for the geodesic ray transform in 2-dimensions \cite{monard2015geodesic},  the restriction of fold-type conjugate points in \cite{SU-folds} was removed to allow all possible types of conjugate locus to arrive at a cancellation of singularities result similar to the one in \cite{SU-folds}. Cancellation of singularities in the presence of conjugate points for the attenuated geodesic ray transform was also considered in \cite{holman2018attenuated}. These works were subsequently generalized by the first author and Uhlmann in \cite{HU} with only the additional assumption of no singular conjugate points and they showed that the normal operator is a sum of a $\Psi$DO and several FIOs with each FIO corresponding to conjugate pairs of a given order. We note that all the works mentioned above are for weighted and non-weighted cases of the geodesic ray transform acting on functions (i.e. the case $m = 0$).

The primary purpose of this paper is to study the geodesic ray transform of symmetric $m$-tensor fields on 2-dimensional Riemannian manifolds in the presence of conjugate points. To the best of our knowledge, ours would be the first work addressing the microlocal analysis of the normal operator associated to the geodesic ray transform acting on symmetric tensor fields on Riemannian manifolds with conjugate points. 
We are interested in undertaking a microlocal study of the transform $I_m$ in this set-up, which is a natural follow-up to \cite{HU}, although that work included no restriction on dimension. In contrast to \cite{HU}, we restrict the dimension of the Riemannian manifold to be $2$ for several reasons. First, due to the restriction on dimension, we place no restrictions on the type of conjugate points. We recall, in dimensions 3 and higher, \cite{HU} required the assumption of no singular conjugate pairs. However, singular conjugate pairs do not occur in 2-dimensions, and so this hypothesis is not necessary for us. Second, as already mentioned above, we explicitly compute the principal symbol of the $\Psi$DO and FIO parts of the operator $I_m^{*}I_m$ in this work, and our computation relies on the stationary phase method \cite{D}. In 2-dimensions, the computation of these symbols is easier as it is a finite sum, in comparison to higher dimensions, where the sum should be replaced by an integral. Finally, in 2-dimensions we can prove a result on cancellation of singularities result similar to \cite{monard2015geodesic}. We hope to address the analysis of the normal operator $I_m^{*}I_m$ acting on symmetric $m$-tensor fields in the presence of conjugate for higher dimensions in a future work.  

In Section \ref{sec:prelim}, we begin with preliminaries and then show in Subsection \ref{sec:Ndecomp} that a similar decomposition of the normal operator into $\Psi$DO and FIO parts to that in \cite{HU} holds for the tensor case. In Section \ref{sec:PDO} we compute explicitly the principal symbol of the $\Psi$DO and FIO components of $I^{*}_mI_m$. Finally, in Section \ref{sec:cancel} we construct a parametrix recovering the solenoidal component of $f$ modulo FIOs, and show a cancellation of singularities result, similar in spirit to the work of \cite{monard2015geodesic}. We emphasize that our cancellation of singularities result is only for interior singularities and we cannot address cancellation of singularities at boundary points with our current techniques. 

%\vkc{Is the above sentence in the correct place?}

%\vkc{Somewhere in the Introduction, we should mention that we cannot address cancellation of singularities results at boundary points.}

%Our primary reasons for the restriction of the study of the geodesic ray transform in 2-dimensions are the following. Due to the restriction on dimension, we place no restrictions on the type of conjugate points. We recall, in dimensions 3 and higher, the strongest result in this direction is the work of \cite{HU} with the additional assumption of no singular conjugate points. Next, as already mentioned above, we explicitly compute the principal symbol of the $\Psi$DO and FIO parts of the operator $I^{*}I$ in this work, and our computation relies on the stationary phase method \cite{D}. In 2-dimensions, the computation of these symbols is easier as it is a finite sum, in comparison to higher dimensions, where the sum should be replaced by an integral. We hope to address the analysis of the normal operator $I^{*}I$ acting on symmetric $m$-tensor fields in higher dimensions in the absence of singular conjugate points in a future work.  

\section{Preliminaries} \label{sec:prelim}
\subsection{The geodesic ray transform}
In this section, we introduce some notation, definitions and results that will be used in this paper. Most of what we give in this section is standard (e.g. see \cite{sharafutdinov2012integral}).

Let $(M,g)$ be a compact 2-dimensional Riemannian manifold with boundary and volume form $\D V_g$. The tangent bundle of $M$ will be denoted by $TM$ with projection map $\pi:TM \rightarrow M$, and $SM\subset TM$ will denote the unit sphere bundle. For points $v \in TM$ we will sometimes use the notation $(x,v)$ which indicates that $x = \pi(v)$. For $v \in TM \setminus \{0\}$, we will write $\gamma_v$ for the maximally extended geodesic with initial condition $\dot{\gamma}_v(0) = v$. Note that this also includes the condition $\gamma_v(0) = \pi(v)$. The outward/inward pointing unit vector fields will be denoted by 
\[
\PD_{\pm}SM=\{v\in SM: 
\pi(v) \in \PD M, \pm\langle \nu,v\rangle_{g}>0\}
\]
where $\nu$ is the unit outer normal on the boundary $\partial M$. The set of smooth functions on $M$ will be denoted $C^\infty(M)$, functions compactly supported in the interior $M^{\rm int}$ of $M$ will be $C_c^\infty(M^{\rm int})$, and distributions on $M$ (the dual space of $C_c^\infty(M)$) will be $\mathcal{D}'(M)$. Distributions that are compactly supported within $M^{\rm int}$ will be denoted $\mathcal{E}'(M^{\rm int})$.

We assume that $(M,g)$ is non-trapping, in the sense that geodesics starting in $M$ will meet the boundary in finite time in both directions. We also assume that the boundary $\PD M$ of $M$ is strictly convex. By this, we mean that the second fundamental form of the boundary defined by 
\[
\Pi(v,v)=\langle \nabla_{v}\nu,v\rangle_{g} \mbox{ for } v\in T(\PD M)
\]
is positive definite. Here $\n$ is the covariant derivative induced by the Levi-Civita connection on $(M,g)$. Since the boundary is strictly convex, the functions $\tau_\pm:SM \rightarrow \mathbb{R}$ which map $v \in SM$ to the positive (resp. negative) time when $\gamma_v$ first intersects $\partial M$ are smooth on the interior of $SM$. \tred{We will also use the notation $D_t$ for the covariant differentiation of a vector field along a curve parameterized by $t$.}

Let $f$ be a symmetric covariant tensor field of rank $m$. In local coordinates $(x^1,\cdots, x^n)$, $f$ can be written as 
\[
f(x)=f_{i_1 \cdots i_m}(x) \D x^{i_1} \cdots \D x^{i_m},
\]
with Einstein summation convention assumed. We write $C^\infty(S^m\tau'_M)$, $C_c^\infty(S^m\tau'_M)$, $\mathcal{D}'(S^m \tau'_M)$, and $\mathcal{E}'(S^m \tau'_{M^{\rm int}})$ for the space of smooth, smooth compactly supported, distribution valued, and compactly supported distribution valued symmetric $m$ tensor fields, respectively. Smooth tensor fields can be considered as functions on the tangent space $TM$, and for this we will use the notation
\[
\langle f, v \rangle = f_{i_1 \cdots i_m} v^{i_1} \cdots v^{i_m}
\]
for any $v \in TM$. The tensor field $f$ being symmetric means that for any permutation $\sigma$ of the indices, 
\[
f_{i_{\sigma(1)} \cdots i_{\sigma(m)}}= f_{i_1 \cdots i_m}.
\]
We will use the usual musical operators to equate covectors and vectors as well as symmetric tensors with indices up and down. For example,
\[
f^{i_1\cdots i_m} = g^{i_1j_1} \cdots g^{i_mj_m} f_{i_1\cdots j_m},
\]
and
\[
\flat_g(v)_i =  v^{j} g_{ij}, \quad \sharp_g(u)^i = u_j g^{ij}.
\]
\tred{The Hodge star operator given by the metric $g$ will be denoted by $\ast_g$. In the 2-dimensional case, $\ast_g$ maps $1$-forms to  $1$-forms by the formula
\[
\ast_g (f)_i = \sqrt{\mathrm{det}(g)} \epsilon_{ji} g^{jl} f_l
\]
where $\epsilon$ is the alternating tensor. Given a covector $\xi \in T^* M$, the Hodge star gives a method to find a vector $\xi^\perp$ annihilated by $\xi$ through the formula $\xi^\perp = \sharp_g(\ast_g(\xi))$. This notation $\xi^\perp$ will be used below.}

We introduce the following pointwise product on tensors
\[
\langle f, h \rangle_g = f_{i_1 \cdots i_m} h^{i_1 \cdots i_m},
\]
and the corresponding $L^2$ inner product on the space of rank $m$ symmetric tensor fields 
\begin{equation}\label{eq:L2}
\langle\langle u,v\rangle\rangle_{g}=\int\limits_{M} \langle u, \overline{v} \rangle_g\ \D V_{g},
\end{equation}
The space $L^2(S^m\tau'_M)$ is the closure of $C^\infty(S^m\tau'_M)$ under the norm corresponding to this inner product. The covariant derivative of a tensor field is given in local coordinates by
\Beq\label{CovariantDerivative}
f_{i_1 \cdots i_{m};k}=\n_k f_{i_1 \cdots i_{m}}= \frac{\PD f_{i_1 \cdots i_m}}{\PD x^{k}}-\sum\limits_{p=1}^{m}\Gamma_{k i_{p}}^{r} u_{i_1 \cdots i_{p-1}r i_{p+1} \cdots i_{m}}
\Eeq
where $\Gamma_{ki}^r$ are the Christoffel symbols. The covariant derivative, and other differential operators introduced below, extend to $\mathcal{D}'(S^m\tau'_M)$ using weak derivatives. Then, the Sobolev spaces $H^k(S^m\tau'_M)$ for integer $k \geq 0$ are the subspaces of $f \in L^2(S^m\tau'_M)$ such that all components of $\nabla^l f$ are in $L^2$ in any coordinates with the natural inner product and norm. For non-integer $s >0$, $H^s(S^m\tau'_M)$ can be defined via interpolation (see \cite[Chapter 4]{taylor1996partial} for more information on the definition of Sobolev spaces).

We introduce an operation $d$ on the space of symmetric tensor fields, called the symmetrized covariant derivative. If $f$ is a smooth symmetric rank $m$ tensor field, then $\D f$ is a rank $m+1$ symmetric tensor field  defined to be the symmetrization of the covariant derivative $\n$. That is, $\D =\sigma \n$ where $\sigma$ is the symmetrization. With respect to the $L^2$ inner product defined in \eqref{eq:L2}, the formal $L^2$ adjoint of $d$ is equal to $-\delta$, where the divergence $\delta$ is defined by  
\[
(\delta f)_{i_1 \cdots i_{m-1}}=  g^{jk}f_{i_1 \cdots i_{m-1} j; k}.
\]
A tensor field $f$ is called {\it solenoidal} if $\delta f = 0$.

In the case of vector fields, $\delta$ is the standard divergence and for this case, the Helmholtz decomposition is classical. A similar decomposition for symmetric tensor fields has played a role in analysis of the tensor ray transform for decades \cite{sharafutdinov2012integral}. Sharafutdinov \cite[Theorem 3.3.2]{sharafutdinov2012integral} proved this decomposition for the case of symmetric tensor fields in $H^k(S^m \tau_M')$ when $k$ is an integer
at least $1$. Although the result in \cite[Theorem 3.3.2]{sharafutdinov2012integral} only explicitly includes $k \geq 1$, it is possible to extend to include $k = 0$ as well with little extra work \cite[Theorem 3.6]{Stefanov_Uhlmann_book}. We will now review the decomposition, sketching its proof.

Suppose that $f \in C^\infty(S^m \tau_M')$ for $k\geq 0$ and consider the boundary value problem
\begin{equation}\label{eq:usys}
\delta d u = \delta f, \quad u|_{\partial M} = 0.
\end{equation}
As demonstrated in \cite{sharafutdinov2012integral}, this is an elliptic boundary value problem which satisfies the Shapiro-Lopatinkii conditions (these are called the complementing condition in \cite{agmon1964estimates}; see also \cite[Proposition 11.9]{taylor1996partial}) and there is a unique solution $u \in C^\infty (S^{m-1} \tau_M')$. Therefore, for $k \geq 1$, the unique solution $u$ satisfies
\begin{equation}\label{eq:uest1}
\|u\|_{H^{k+1}(S^m\tau_M')} \leq C \|\delta d u\|_{H^{k-1}(S^m\tau_M')} = C \|\delta f \|_{H^{k-1}(S^m\tau_M')} \leq C \|f\|_{H^k(S^m\tau_M')}.
\end{equation}
Note that the standard Shapiro-Lopatinskii result gives the first step of this estimate only for $k \geq 1$. To obtain the same estimate for $k = 0$, we have
\[
\|du\|_{L^2(S^m\tau'_M)}^2 = \langle f , d u \rangle_{L^2(S^m\tau'_M)} \leq \|f\|_{L^2(S^m\tau'_M)} \|du \|_{L^2(S^m\tau'_M)} \Rightarrow \|du\|_{L^2(S^m\tau'_M)} \leq \|f\|_{L^2(S^m\tau'_M)}.
\]
The first equality above follows from the weak formulation of \eqref{eq:usys}. Using the equality
\begin{equation}\label{eq:P1}
\langle u(\gamma_v(t)), \dot{\gamma}_v(t) \rangle = \int_0^t \langle du(\gamma_v(s)), \dot{\gamma}_v(s) \rangle \ ds
\end{equation}
which holds for $v \in \partial_- SM$ since $u |_{\partial M} = 0$, we can establish a Poincar\'e-type inequality for $u$. Indeed, using Santal\'o's formula and \eqref{eq:P1}
\[
\begin{split}
\|u\|_{L^2(S^m\tau'_M)}^2 & \leq C \int_{SM} \langle u(x),v \rangle^2 \ dSM(x,v)  = \int_{\partial_- SM} \int_0^{\tau_+(v)} \langle u(\gamma_v(t)),\dot{\gamma}_v(t) \rangle^2 \ d t\  d \partial_- SM(v)\\
& 
\leq C\int_{\partial_- SM} \int_0^{\tau_+(v)} \left (\int_0^t \langle du(\gamma_v(s)), \dot{\gamma}_v(s) \rangle \ ds \right )^2 \ d t\  d \partial_- SM(v).
\end{split}
\]
By the Cauchy-Schwartz inequality and the fact that $\tau_+(v)$ is bounded by hypothesis, we then have
\[
\|u\|_{L^2(S^m\tau'_M)}^2 \leq C \int_{\partial_- SM} \int_0^{\tau_+(v)} \langle du(\gamma_v(s)), \dot{\gamma}_v(s) \rangle^2  \ d s\  d \partial_- SM(v)
\]
for a new constant $C$. Applying Santal\'o's formula a second time then proves 
\[
\|u\|_{L^2(S^m \tau'_M)} \leq C \| d u \|_{L^2(S^m \tau'_M)}.
\]
Note the constant $C$ may have changed from the previous step. Using the generalized Korn's inequality as well (\cite[Theorem 2.2]{pompe2003korn}), implies
\begin{equation}\label{eq:uest}
\|u\|_{H^{k+1}(S^m\tau'_M)} \leq C \|f\|_{H^k(S^m\tau_M')}
\end{equation}
for $k = 0$. Considering also \eqref{eq:uest1}, we have established \eqref{eq:uest} for all integers $k \geq 0$.

Given $u$ the unique solution of \eqref{eq:usys}, if we define
\[
^s\!f = f - du,
\]
then $^sf$ is solenoidal and we obtain the solenoidal decomposition
\[
f =  {^s}\!f + du.
\]
Note additionally by \eqref{eq:uest} that
\begin{equation}\label{eq:sfest}
\|^s\!f\|_{H^k(S^m\tau'_M)} \leq C \|f\|_{H^k(S^m\tau'_M)}
\end{equation}
for any $k \geq 0$. By continuity, we can extend the above construction to allow $f \in L^2(S^m\tau'_M)$, in which case $^s\!f$ and $du \in L^2(S^m\tau'_M)$ are projections onto the space of solenodal tensor fields and potential tensor fields, respectively. These spaces are orthogonal in $L^2(S^m\tau'_M)$, which leads to
\[
\|f\|^2_{L^2(S^m\tau'_M)} = \|^s\! f\|_{L^2(S^m\tau'_M)}^2 + \|du\|_{L^2(S^m\tau'_M)}^2.
\]
By interpolation, we also obtain \eqref{eq:uest} and \eqref{eq:sfest} with $k$ replaced by any $s\in \mathbb{R}$ with $s\geq 0$.

The arguments above result in the Helmholtz decomposition for symmetric tensor fields which we now record in a theorem.

\begin{theorem}\label{thm:Solenoidal}
    Let $\delta$ and $d$ be the differential operators defined above with domains $D(\delta) = L^2(S^m\tau'_M)$ and $D(d) = H^1_0(S^m\tau'_M)$. Then the space $L^2(S^{m}\tau_M')$ admits an orthogonal decomposition
    \[
    L^2(S^m \tau'_M) = \operatorname{Ker}(\delta)\oplus\operatorname{Ran}(d),
    \]
    known as the solenoidal-potential decomposition. If, $f \in H^s(S^m\tau'_M)$ for $s \geq 0$ and
    \[
    f =\, ^s\!f + du
    \]
    is the solenoidal-potential decomposition of $f$, then
    \[
    \|^s\! f\|_{H^s(S^m\tau'_M)} \leq C \| f \|_{H^s(S^m\tau'_M)}
    \]
    and
    \[
    \|u\|_{H^{s+1}(S^m\tau'_M)} \leq C \| f\|_{H^s(S^m\tau'_M)}
    \]
    for a constant $C>0$ independent of $f$.
\end{theorem}

\noindent We will write $\pi_S$ for the projection of $L^2(S^m\tau'_M)$ onto the solenoidal part as described in Theorem \ref{thm:Solenoidal}.
%\vkc{To be consistent: Should we use $\pi_s$ to denote the solenoidal projection?(small $s$ instead of big $S$).}
% The following decomposition theorem is due to Sharafutdinov. 

% \begin{theorem}
% Let $(M,g)$ be a compact Riemannian manifold with boundary. Let $k\geq 1$ and $m\geq 0$ be integers. Let $f$ be a symmetric $m$-tensor field with $f\in H^{k}(M)$, where $H^{k}$ is the standard Sobolev space on $M$. Then there exists uniquely determined symmetric $m$-tensor field $f^{s}\in H^{k}(M)$ and another symmetric $m-1$ tensor field  $v\in H^{k+1}(M)$ such that 
% \[
% f=f^{s}+ \D v, \quad \delta f^{s}=0, \quad v|_{\PD M}=0.
% \]  
% \end{theorem}
% In the theorem above $-\delta$ is the formal $L^2$ adjoint of $\D$ - the symmetrized covariant derivative. The components $f^s$ is called the solenoidal component of $f$ and $\D v$ is the potential component. 

We now precisely define the ray transform of symmetric $m$-tensor fields as follows: It is initially defined on the space of smooth symmetric $m$-tensor fields
\[
I_m: C^{\infty}(S^{m}\tau_{M}')\to C^{\infty}(\PD_{+}SM)
\]
as 
\begin{equation}\label{eq:I}
I_mf(v)=\int\limits_{\tau_{-}(v)}^{0} \langle f(\gamma_v(t)), \dot{\gamma}_v(t)\rangle\, \D t.
%f_{i_1 \cdots i_{m}}(\g_{x,v}(t))\dot{\g}^{i_1}_{x,v}(t) \cdots \dot{\g}^{i_m}_{x,v}(t) \, \D t.
\end{equation}
%Here $\g_{v}(t): [\tau_{-}(v),0]\to M$ is the maximal geodesic starting at $x$ in the direction $v$.
To define the backprojection operator, we introduce the mapping $F:SM \rightarrow \partial_+ SM$ given by\begin{equation}\label{eq:F}
F(v) = \dot{\gamma}_{v}(\tau_+(v)).
\end{equation}The backprojection operator $I_m^{*}: C^{\infty}(\PD_{+} SM)\to C^{\infty}(S^{m} \tau'_{M})$ is then defined in local coordinates by 
\begin{equation} \label{eq:I*}
(I_m^{*} h)_{i_1 \cdots i_m}(x)=\int\limits_{S_{x}M} h(F(v)) v_{i_1} \cdots v_{i_m} \D S_{x}(v).
\end{equation}
Note that $I_m^*$ is the formal adjoint of $I_m$ with respect to the $L^2$ inner product \eqref{eq:L2}. Additionally, both $I_m$ and $I_m^*$ can be extended as continuous maps on the respective $L^2$ spaces, and to compactly supported distribution spaces.

\subsection{Fourier integral operator and its principal symbol}

Our main results give the principal symbols of the pseudodifferential and Fourier integral parts of the operator $\mathcal{N}_m$. Given this, we now introduce Fourier integral operators (FIOs). As a first reference on the topic we recommend \cite{D}, although this only considers FIOs acting on scalars. Because $\mathcal{N}_m$ acts on tensor fields, we must also consider FIOs acting on tensor fields, or, more generally, Fourier integrals taking values in a vector bundle. For a general treatment of FIOs, which includes this case, see \cite{hormanderIV}.

Let $X$ and $Y$ be smooth manifolds of dimension $n_X$ and $n_Y$ respectively, and with non-vanishing densities $|\D v_X|$ and $|\D v_Y|$. Let $E_X$ and $E_Y$ be smooth vector bundles over $X$ and $Y$ respectively with $E_X$ having a fiberwise metric $\langle \cdot, \cdot \rangle_{E_X}$.  For our purposes, $E_X$ and $E_Y$ will be the symmetric tensor bundles $S^m \tau'_M$, but let us continue to consider the general case for now. Finally, let $\Cc \subset T^*X \setminus \{0\} \times T^* Y \setminus \{0\}$ be a canonical relation (see \cite{D} for the definition of canonical relation). An FIO of order $m$ is an operator $A: C^\infty_c(E_Y)\to \Dc'(E_X)$ such that the Schwartz kernel $K_{A}$ of $A$ can be locally represented in the form
\begin{equation} \label{eq:KAdef}
K_A(x,y)=\int_{\mathbb{R}^n} e^{\I \vp(x,y,\theta)}a(x,y,\theta) \D \theta\ |\D v_Y|^{1/2}\ |\D v_X|^{1/2},
\end{equation}
where $\vp(x,y,\theta)$ is a non-degenerate phase function which locally defines $\Cc$ and $a$ is an amplitude function of order $\mu = m + n/2-(n_X+n_Y)/4$ taking values in $\mathrm{Hom}(E_Y,E_X^*)$ ($E_X^*$ is the dual bundle to $E_X$). The technical terms in the previous sentence will be explained below.

\begin{remark}
This definition differs slightly from \cite[Section 25.2]{hormanderIV}, where FIOs acting on sections of vector bundles are identified with distributions in $\mathcal{D}'(X \times Y, \Omega_{1/2} \otimes \mathrm{Hom}(E_Y,E_X))$. However, it is not clear how to identify such distributions with operators from sections of $E_Y$ to sections of $E_X$ without some additional structure. We have chosen our convention, $a$ taking values in $\mathrm{Hom}(E_Y,E_X^*)$, to avoid this. Indeed, for $f \in C_c^\infty(E_Y)$ and $h \in C_c^\infty(E_X)$, we have
\[
\langle K_A, (h |\D v_Y|^{1/2}) \otimes f |\D v_X|^{1/2} \rangle = \int_X \int_Y \int_{\mathbb{R}^n} e^{i \varphi(x,y,\theta)} \langle a(x,y,\theta) f, h \rangle\ \D \theta|\D v_Y| |\D v_X|.
\]
where $\pi_X$ and $\pi_Y$ are projections onto the respective components. The relationship between $K_A$ and $A$ is given, for $f \in C_c^\infty(\Omega_{1/2}\otimes E_Y)$ and $h \in C_c^\infty(\Omega_{1/2}\otimes E_X)$, by
\[
\langle K_{A}, h \otimes f \rangle = \langle Af, h \rangle.
\]
These relationships are required to make sense of Definition \ref{def:prin}. 
\end{remark}

Non-degeneracy of the phase funciton $\vp$ means that it is a real-valued, homogeneous function of degree $1$ in $\theta$, $\D_{x,y,\theta}\vp\neq 0$ for all $(x,y,\theta) \in X\times Y \times \Rb^{n}\setminus \{0\}$ and at the set of points $(x,y,\theta)$ such that $\PD_{\theta} \vp(x,y,\theta)=0$, the set $\{ \frac{\PD}{\PD \theta_{j}}\lb \PD_{x,y,\theta} \vp\rb (x,y,\theta)\}$ for $1\leq j\leq n$ is linearly independent. The phase function $\vp$ defines $\Cc$ in the sense that,
locally on $\Cc$,
\begin{equation}\label{eq:Lamdef}
\Cc = \Big \{ (\D_x \vp(x,y,\theta),-\D_y \vp (x,y,\theta)) \in T^* X\setminus \{0\} \times T^* Y \setminus \{0\} \ : \ \D_\theta \vp (x,y,\theta) = 0 \Big \}.
\end{equation}
The function $a$ is called an amplitude function and it satisfies an amplitude estimate, depending on its order $\mu$, defined as formula 
\eqref{eq:ampest} below. Note that \eqref{eq:KAdef} is of the form
\begin{equation} \label{eq:Fint}
K_A(z) = \int_{\mathbb{R}^n} e^{i \vp(z,\theta)} a(z,\theta) \D \theta\ |\D v_Z|^{1/2},
\end{equation}
where $z = (x,y) \in X \times Y = Z$ and $|\D v_Z|$ is the product density. We call this type of integral a Fourier integral and will introduce the next few concepts for expressions of this form. The phase function $\vp$ in \eqref{eq:Fint} also defines a Lagrangian manifold by
\[
\Lambda = \Big \{ \D_z\varphi(z,\theta) \in T^* Z \setminus \{0\} \ : \ \D_\theta \varphi(z,\theta) = 0 \Big \}.
\]
Note that, when \eqref{eq:Fint} is given by an FIO as described above, then the canonical relation $\Cc$ given by \eqref{eq:Lamdef} is related to $\Lambda$ through multiplication by $-1$ in the $T^*Y$ component.

For a vector bundle $E_Z$ on $Z$, a function $a\in C^{\infty}(Z\times \Rb^n;E_Z)$ is said to be an amplitude function if for each compact $K\subset Z$ and multi-indices $\A, \B$, we have the following estimate: 
\begin{equation}\label{eq:ampest}
\lvert \PD_z^{\B} \PD_\theta^{\A} a(z,\theta)\rvert \leq C_{K,\A}(1+|\theta|)^{\mu-|\A|}
\end{equation}
for some real $\mu$. The norm on the left side of this estimate can be taken to be any fibrewise norm on the space $E_Z$ in local coordinates. The space of functions satisfying an estimate of the above form will be denoted $S^{\mu}(Z\times \Rb^n;E_Z)$ and $\mu$ is called the order of the amplitude function. For the application of FIOs, we will take $E_Z = \mathrm{Hom}(E_Y,E_X^*)$, or more specifically in our case $E_Z = \mathrm{Hom}(S^m\tau'_M,(S^m\tau'_M)^*)$. It is also possible to define amplitudes on a conic manifold taking values in a vector bundle. In particular, $S^\mu(\Lambda;\Omega_{1/2} \otimes M_\Lambda)$ is the space of amplitudes defined on $\Lambda$ with values in the product bundle of the set of half-densities $\Omega_{1/2}$ and the Keller-Maslov line bundle $M_\Lambda$ over $\Lambda$. For details of these definitions, see \cite{D}.

% \shc{Note, this is an important point because the principal symbols are in such spaces. Maybe we should say more?}
% \vkc{With this as our definition, the principal symbols we have obtained are correct right? In the sense that the indices are all up as they should be. However, this definition differs from Hormander definition. Isn't it? He defined it as $Home(E,F)$, unless I am miunderstanding something.}

% Before proceeding, it is worthwhile to consider the precise definition of the Schwartz kernel $K_A \in \mathcal{D}'(E_X \otimes E_Y)$ in \eqref{eq:KAdef}. For $f \in C_c^\infty(E_X \times E_Y)$, we have
% \[
% \langle K_A, f \rangle = \int_X \int_Y \int_{\mathbb{R}^n} e^{i \varphi(x,y,\theta)} \langle a(x,y,\theta) \pi_{E_Y}(f), \pi_{E_X}f \rangle\ \D \theta|\D v_Y| |\D v_X|
% \]
% where $\pi_X$ and $\pi_Y$ are projections onto the respective components. The relationship between $K_A$ and $A$ is given, for $f \in C^\infty(E_Y)$ and $h \in C^\infty(E_X)$, by
% \[
% \langle K_{A}, h \otimes f \rangle = \langle Af, h \rangle_{E_X}.
% \]
% These relationships are required to make sense of Definition \ref{def:prin}.
%We then say that $A$ is a Fourier integral operator defined by a non-degenerate phase function $\vp$ and an amplitude function $a$. We will use the term Fourier integral to denote the Schwartz kernel, $K_A$ defined above, of an FIO.
%A Fourier integral operator is a locally finite 

In the sections below, we use the method of stationary phase to compute the principal symbols of Fourier integrals, and therefore the principal symbol for corresponding FIOs. We now state the definition for the principal symbol of a Fourier integral. 
\begin{definition}\cite[Definition 4.1.1]{D} \label{def:prin}
    Let $Z\subset \Rb^{n_Z}$ be open and let $K_A$ be a Fourier integral of order $m$ defined by a non-degenerate phase function $\vp$ and an amplitude $a\in S^{m-(n/2)+(n_Z/4)}(Z\times \Rb^n;E_Z)$ as in \eqref{eq:Fint}. Let $\psi(z,\alpha):Z \times \Lambda \rightarrow \mathbb{R}$ be smooth, homogeneous of order $1$ in $\alpha$, and such that, for fixed $\alpha$, $\{d_x\psi(x,\alpha) \} \subset T^* Z$ intersects $\Lambda$ transversely at $\alpha$. Suppose $p_L: T_\alpha(T^*Z) \rightarrow \mathrm{Ker}(D \pi|_\alpha)$ is the projection along $T_\alpha \D\psi$ restricted to $T_\alpha\Lambda$, and $\omega$ is the volume form in $\mathrm{Ker}(D \pi|_\alpha)$. Finally, let $u \in C^\infty_c(Z;\Omega_{1/2} \otimes E_Z^*)$. The principal symbol of $K_A$ is the element in 
    \[
    S^{m+(n_Z/4)}(\Lambda, \O_{1/2}\otimes M_{\Lambda}\otimes E_Z)/S^{m+(n_Z/4)-1}(\Lambda, \O_{1/2}\otimes M_{\Lambda} \otimes E_Z)
    \]
    given by 
    \[
    E_Z^* \ni u(\pi(\alpha)) \to e^{\I \psi(\pi(\A),\A)}\langle K_A, u e^{-\I \psi(x,\A)}\rangle |p_L^* \omega|^{1/2}, 
    \]
    %\shc{There is some confusion here between ``Fourier Integral" and ``Fourier Integral Operator". (i.e. the Schwartz kernel of an FIO is a Fourier integral}
   % \vkc{This has been addressed above.}
    for $\A\in \Lambda$.% (where $\Lambda$ is the conic Lagrangian manifold defined by the non-degenerate phase function $\vp$). Here $S^{\mu}(\Lambda, \O_{1/2}\otimes L)$ denotes the symbol space of sections of the complex line bundle $\O_{1/2}\otimes L$ over $\Lambda$ of order $\mu$. Moreover $u\in C_c^{\infty}(X,\O_{1/2})$ and  $\psi(x,\A)$ is homogeneous of order $1$ in $\A$ and $\mbox{Graph}(x\to \PD_{x} \psi(x,\A))$ intersects $\Lambda$ transversely at $\A$. 
\end{definition}

When we apply Definition \ref{def:prin} to the case of FIOs, we identify the Lagrangian $\Lambda$ with the canonical relation $\Cc$ by changing the sign of the $T^*Y$ coordinate. Note that the principal symbol of an FIO will take values in the vector bundle $\Omega_{1/2}\otimes M_{\Cc} \otimes \mathrm{Hom}(E_Y,E_X^*)$ over $\Cc$.

\subsection{Geodesic ray transform as an FIO} \label{sec:Ndecomp}

It is well-known that $I_m$ and $I_m^*$ are Fourier integral operators of order $-1/2$. For the case $m=0$ (the scalar transform) this can be shown using the fact that the Schwartz kernel of $I_m$ is a weighted delta type distribution in the space $\partial_+ SM \times M$ \cite{monard2015geodesic}, and the same follows follows for each component of the tensor transform. We will discuss the canonical relations of $I_m$ and $I_m^*$ below (see \eqref{eq:CI} and \eqref{eq:CI*}), but first introduce some general considerations useful for analysis of the normal operator
\begin{equation}\label{eq:normal}
\mathcal{N}_m = I_m^* I_m.
\end{equation}
In \cite{HU}, the authors consider the push-forward and pull-back of half-densities by a mapping $G:X \rightarrow Y$ between manifolds with dimensions $n_X$ and $n_Y$, respectively. We will use a similar definition, but without the complication of half-densities.

\begin{definition}
    Let $G:X \rightarrow Y$ be a smooth submersion between manifolds which are given densities $|\D \mu_X|$ and $|\D \mu_Y|$. Then the push-forward $G_*:C_c^\infty(X) \rightarrow C_c^\infty(Y)$ and pull-back $G^*: C_c^\infty(Y) \rightarrow C^\infty(X)$ are defined by
    \[
    \int\limits_Y h(y)G_*[f](y) |\D \mu_Y(y)| = \int\limits_X h(G(x)) f(x) |\D \mu_X(x)| = \int\limits_X G^*[h](x) f(x) |\D \mu_X(x)|.
    \]
\end{definition}

The result \cite[Lemma 1]{HU} remains true in this situation and we have that $G_*$ and $G^*$ are both FIOs of order $(n_Y - n_X)/4$. We will require this result for the case of $F$ defined by \eqref{eq:F}.
For this mapping, we have $F_*$ and $F^*$ are FIOs of order $-1/4$ and canonical relations given by
\begin{equation}\label{eq:F_*}
\Cc_{F_*} = \left \{ (\xi, DF|_{v}^t\xi) \ : \ v \in SM, \ \xi \in T^*_{F(v)} \partial_+ SM \setminus\{0\}\right \}, 
\end{equation}
\begin{equation}\label{eq:F^*}
\Cc_{F^*} = \left \{ (DF|_{v}^t\xi,\xi) \ : \ v \in SM, \ \xi \in T^*_{F(v)} \partial_+ SM\setminus\{0\}\right \}.
\end{equation}
Unlike the situation in \cite{HU}, which considered the geodesic ray transform on scalar fields, the operators $I_m$ and $I^*_m$ cannot be written as the composition of a pull-back and a push-forward. However, we can modify the framework of \cite{HU} slightly by introducing the operator $\pi^*: C^\infty(S^m \tau'_M) \rightarrow C^\infty(SM)$ defined by
\begin{equation}\label{eq:pi^*}
    \pi^*[h](v) = \langle h, v \rangle
\end{equation}
with adjoint given, in local coordinates, by
\begin{equation}\label{eq:pi_*}
    \pi_*[f](x) = \int\limits_{S_x M} f(x,v) v_{i_1} \cdots v_{i_m} | \D S_xM(v)|,
\end{equation}
which maps $\pi_*:C^\infty(SM) \rightarrow C^\infty(S^m \tau'_M)$. %Though \eqref{eq:pi_*} is in local coordinates, the operator is invariantly defined.
We note that this is an abuse of notation since $\pi^*$ and $\pi_*$ defined by \eqref{eq:pi^*} and \eqref{eq:pi_*} are not the push-forward and pull-back of the projection $\pi:SM \rightarrow M$, although they play the same role that the push-forward and pull-back did in \cite{HU}.

Using Santal\'o's formula, we have, 
\begin{equation}\label{eq:II*}
I_m = F_* \circ \pi^*, \ I_m^* = \pi_* \circ F^*
\end{equation}
by a calculation very similar to that found in \cite{HU}. We would also like to be able to apply \cite[Lemma 1]{HU} to conclude $\pi_*$ and $\pi^*$ are FIOs, but it does not apply because these are not the push-forward and pull-back of a submersion. However, we have the identity
\[
\int\limits_{M} \langle h, \pi_*[f]\rangle_g \ \D V_g = \int\limits_{SM} \langle h, v\rangle f \ |\D SM| = \int\limits_{SM} \pi^*[h] f \ | \D SM|.
\]
If we take coordinates $\{x^1,x^2,\alpha\}$ on $SM$, where $\{x^1,x^2\}$ are coordinates on $M$, then if $f$ is supported within this coordinate neighborhood the integrals above are equal to
\begin{equation}\label{eq:FIOcoord}
\int\limits_{\mathbb{R}^3_{(x,\alpha)} \times \mathbb{R}^2_y\times \mathbb{R}^2_\xi} e^{i\eta \cdot (x-y)} h_{i_1 \cdots i_m}(y) v^{i_1}\cdots v^{i_m} f(x,\alpha) J(x,\alpha) \ \D x\D \alpha \D y \D \eta
\end{equation}
where $J(x,\alpha) \D x \D \alpha = |\D SM|$ in the coordinate domain. Note that $v$ in this equation depends on $\alpha$. From \eqref{eq:FIOcoord}, we see that
$\pi_*$ and $\pi^*$ are FIOs of order $-1/4$ with canonical relations
\begin{equation}\label{eq:Cpi_*}
\Cc_{\pi_{*}}=\Big{\{}(\eta; D\pi|^t_{w} \eta): w \in SM, \eta\in T^{*}_y M\setminus \{0\}\Big{\}},
\end{equation}
\begin{equation}\label{eq:Cpi^*}
\Cc_{\pi^{*}}=\Big{\{}(D\pi|^t_{w} \eta;\eta): w\in SM, \eta\in T^{*}_y M\setminus \{0\}\Big{\}},
\end{equation}
It is possible to use \eqref{eq:F_*}, \eqref{eq:F^*}, \eqref{eq:Cpi_*}, \eqref{eq:Cpi^*} and \eqref{eq:II*} to determine the canonical relation of $I_m$ and $I_m^*$ by composition. Indeed, this gives
\begin{equation}\label{eq:CI}
\Cc_{I_m} =  \left \{ (\xi;\eta) \ : \ w \in SM,\ \eta \in  T^{*}_y M\setminus\{0\}, \ D\pi|_{w}^t \eta = DF|_{F(w)}^t \xi \right \},
\end{equation}
\begin{equation}\label{eq:CI*}
\Cc_{I_m^*} =  \left \{ (\eta;\xi) \ : \ w \in SM,\ \eta \in  T^{*}_y M\setminus\{0\}, \ D\pi|_{w}^t \eta = DF|_{F(w)}^t \xi \right \}.
\end{equation}
These canonical relations can also be written as a normal bundle and are local diffeomorphisms \cite[Theorem 3.2]{monard2015geodesic}. 

The operator $I^*_m$ fails to satisfy the Bolker condition globally because geodesics passing normally to $\eta$ in either direction will give rise to different $\xi$ in \eqref{eq:CI*}. However, if we localize by taking a cut-off function $\varphi \in C_c^\infty(\partial_+SM)$ which is supported near a single point in $\partial_+SM$, then $I_m^* \varphi$ (by which we mean the operator localized on the right) does satisfy the Bolker condition \cite{GS}. In fact, since we consider the two dimensional case, the canonical relations $\Cc_{I_m}$ and $\Cc_{I_m^*}$ are local diffeomorphisms by \cite[Theorem 3.2]{monard2015geodesic}.  %\shc{Does this last sentence make sense? I am trying to minimize proliferation of notation. We could also include this as a Lemma, and later I am planning to prove that the corresponding normal operator $(I^*\varphi)^* I^* \varphi$ is elliptic. Also, do we need a reference for the Bolker condition?}

The canonical relations \eqref{eq:Cpi_*} and \eqref{eq:Cpi^*} are the same as the canonical relations for the push-forward and pull-back of the projection map $\pi$ considered in \cite{HU}, and therefore the same analysis can be applied to the normal operator \eqref{eq:normal}. Indeed, since we are considering the two dimensional case when there are no singular conjugate pairs, and the same arguments which prove \cite[Theorem 4]{HU} show that
\begin{equation}\label{eq:Ndecomp}
\mathcal{N}_m = I_m^* I_m = \Psi_m + \sum_{j=1}^k A_{m,j}
\end{equation}
where $\Psi_m$ is a pseudodifferential operator and for each $j$, $A_{m,j}$ is an FIO with canonical relation given by one connected component of the set
\begin{equation}\label{Lj}
\begin{split}
\Cc_{J} & = \Bigg \{ \Big (\gamma(a),\flat_g(D_t J(a)); \gamma(b), \flat_g(D_t J(b))\Big ) \\ &\hskip2cm : \ \mbox{$J$ is a Jacobi field along $\gamma$ with $J(a) = J(b) = 0$} \Bigg \}.
\end{split}
\end{equation}

\begin{comment} 
\section{Analysis of $F^* \circ F_*$}

From \cite{HU}, we have the explicit formula
\[
\mathcal{N}_{SM}[f](x,v) = F^* \circ F_*[f](x,v) = \int_{\tau_-(x,v)}^{\tau_+(x,v)} \frac{h}{|\D SM|^{1/2}}
(\dot{\gamma}_{(x,v)}(s))\ \D s \ |\D SM(x,v)|^{1/2}.
\]
Let us suppose that $h$ is compactly supported within the domain of a single coordinate chart on $SM$ having coordinates $(y,w)$ and we consider $(x,v)$ in local coordinates as well. Then using the Fourier inversion formula we have
\[
\mathcal{N}_{SM}[f](x,v) = \frac{1}{(2 \pi)^{2n-1}} \int_{-\infty}^\infty \int_{\Rb^{2n-1}} \int_{\Rb^{2n-1}} e^{i(\dot{\gamma}_{(x,v)}(s)-(y,w)))\cdot (\xi,\eta)} \frac{h}{|\D SM|^{1/2}}(y,w) \ \D y \D w \ \D \xi \D \eta \ \D s\ |\D SM(x,v)|^{1/2}.
\]
\end{comment} 

\section{Principal symbols of the pseudodifferential and Fourier integral components of the normal operator} \label{sec:PDO}

In this section, we will calculate the principal symbols of the diagonal and conjugate parts of the normal operator $\mathcal{N}_m$ using Definition \ref{def:prin}. From \eqref{eq:Ndecomp}, we see that $\mathcal{N}_m$ decomposes into a sum of pieces the first of which is a $\Psi$DO, %\tred{and away from singular conjugate pairs,} 
%\vkc{This phrase is not needed.}
while the other components are FIOs each with canonical relation given by a component of $\Cc_J$, which is given by \eqref{Lj}. The principal symbol of $\mathcal{N}_m$ at a point $
\alpha = (x_0,\xi;y_0,\eta)$ either in the diagonal $\Cc = \Delta$ for the $\Psi$DO part or in $\Cc = \Cc_J$ for the FIO parts, is given by the leading order term in the asymptotic expansion of 
\begin{equation}\label{psymb}
\phi(\pi(\alpha)) \mapsto e^{i \psi(x_0,y_0,\alpha)} \langle K_{\mathcal{N}_m},e^{-i\psi(x,y,\alpha)} \phi(x,y) \rangle |p_L^* \omega|^{1/2}
\end{equation}
where $K_{\mathcal{N}_m}$ is the Schwartz kernel of $\mathcal{N}_m$, $\psi \in C^\infty(M^{\rm int}\times M^{\rm int} \times \Cc)$ is homogeneous of order 1 in $\Cc$ and for fixed $\alpha$ has differential graph which intersects $\Lambda$ transversally at $\alpha'$. Here $\Lambda$ and $\alpha'$ are given by
\[
\Lambda = \{(x,\xi;y,\eta) \ : \ (x,\xi;y,-\eta)\in \Cc \}, \ \alpha' = (x_0,\xi;y_0,-\eta).
\]
Also, $\phi$ is a half density on $M^{\mathrm{int}} \times M^{\mathrm{int}}$ with values in $S^m\tau'_M \otimes S^m \tau'_M$ and $|p_L^* \omega|^{1/2}$ is from Definition \ref{def:prin}.

We first compute the principal symbol of the $\Psi$DO part of the operator $\Nc_m$. This is given by the following theorem. 
\begin{theorem}\label{Thm:PS:PsiDO} The principal symbol 
$\sigma_m$ of the operator $\Psi_m$ in \eqref{eq:Ndecomp} is
\Beq\label{PS:PsiDO}
\begin{aligned}
   \sigma_m^{i_1 \cdots i_m j_1 \cdots j_m}(x_0,\xi, x_0,\xi)&\\
   &\hskip-2cm = \frac{(2\pi )^2 i\lb \xi^{\perp}_{\tred{0}}\rb 
   ^{i_1} \cdots \lb \xi^{\perp}_{\tred{0}}\rb^{i_m}\lb \xi^{\perp}_{\tred{0}}\rb^{j_1}\cdots \lb \xi^{\perp}_{\tred{0}}\rb^{j_m}}{|\xi|_\tred{g}}\lb \frac{1}{ \langle w_0, \nu (z_0) \rangle_g}+ \frac{1}{ \langle w_1, \nu (z_1) \rangle_g}\rb \tred{|\D x \D \xi|^{1/2}}, 
\end{aligned}
\Eeq
with \tred{$\xi^\perp_0 = \xi^\perp/|\xi^{\perp
}|_g$}, $z_0, z_1$ and $w_0, w_1$ given in \eqref{z0z1-Expressions} and \eqref{w0w1-Expressions}, respectively. The vector field $\nu$ is the unit outer normal at the boundary.
%\shc{I am changing this to be consistent with $\nu$ in the preliminaries section, which is defined as the outward unit normal. Leaving this note for now as more changes are likely required below as well.}
Furthermore, $\Psi_m$ is elliptic on solenoidal tensor fields, in the sense that, if $f$ is a symmetric $m$-tensor such that 
\[
\xi^{i} f_{i_1 \cdots i_{m-1}i}(\xi)=0 \mbox{ and } \sigma_m f=0, 
\]
then these imply $f = 0$. 
\end{theorem} 

\tred{\begin{remark}
    Since $\Psi_m$ is a pseudodifferential operator, it is possible to identify $\Cc = \Delta$ with $T^* M$, parameterized by $(x,\xi)$, and divide \eqref{PS:PsiDO} by the half density corresponding to the identity operator to obtain a section of bundle $(S^m\tau'_M)^* \otimes (S^m \tau'_M)^*$ over $T^* M$. This is often referred to as the principal symbol of a pseudodifferential operator. Since the principal symbol of the identity operator is
    \[
    2\pi\I |\D x \D \xi|^{1/2},
    \]
    for our case we would obtain the alternate representation of the principal symbol
    \[
    \sigma_m^{i_1 \cdots i_m j_1 \cdots j_m}(x_0,\xi, x_0,\xi) = \frac{2\pi  \lb \xi^{\perp}_{\tred{0}}\rb 
   ^{i_1} \cdots \lb \xi^{\perp}_{\tred{0}}\rb^{i_m}\lb \xi^{\perp}_{\tred{0}}\rb^{j_1}\cdots \lb \xi^{\perp}_{\tred{0}}\rb^{j_m}}{|\xi|_g}\lb \frac{1}{ \langle w_0, \nu (z_0) \rangle_g}+ \frac{1}{ \langle w_1, \nu (z_1) \rangle_g}\rb.
    \]
\end{remark}}

%We first consider the principal symbol of the pseudodifferential part of $\mathcal{N}$. For this calculation, 
\bpr
We work in a set of normal coordinates $\{ x^1,x^2\}$ with domain $U$. In these coordinates, we write $\alpha = (x_0,\xi;x_0,\xi)$ for points in $\Delta$ and we suppose without loss of generality that the coordinates $\{x^1,x^2\}$ are centered at $x_0$ and so $x_0 = 0$. Take any $\phi \in C^\infty_c(U)$ \tred{such that $\phi(x_0) = 1$} and let us define two symmetric $m$-tensor fields $f,g \in  C_c^\infty(U; S^{m}\tau'_U)$ by
\[
f(x,\alpha) =  \phi(x) e^{-i \langle x, \xi \rangle} \D x^{i_1}\ \cdots \D x^{i_m}|_{x}, \quad g(y,\alpha) = \phi(y) e^{i\left ( \langle y, \xi \rangle - \frac{h(\xi)}{2}|y|^2 \right )}\lb \D x^{j_1} \ \cdots \ \D x^{j_m}\rb |_{y}
\]
where the covectors and exponentials are defined using the same normal coordinates on $U$ and
\[
\D x^{i_1} \cdots \D x^{i_m}=\frac{1}{m!}\sum\limits_{\sigma\in \Sigma_m} \D x^{i_{\sigma(1)}}\otimes \cdots \otimes \D x^{i_{\sigma(m)}}.
\]
Here $\Sigma_m$ is the permutation group corresponding to permutation of \{$1,\cdots, m\}$. Note that each index $i_l$ is either $1$ or $2$ in our setup. Thus 
\begin{equation}\label{eq:fg1}
f(x,\alpha) \otimes g(y,\alpha) = \phi(x)\phi(y) e^{-i\left ( \langle x-y, \xi \rangle + \frac{h(\xi)}{2}|y|^2\right )}\lb \D x^{i_1} \ \cdots \ \D x^{i_m}\rb|_{x}  \otimes \lb \D x^{j_1} \ \cdots \ \D x^{j_m}\rb|_{y}.
\end{equation}
Then applying $K_{\mathcal{N}_m}$ to $(f |\D V_g|^{1/2}) \otimes (g |\D V_g|^{1/2})$ as in \eqref{psymb} we have
\begin{equation}\label{psiphi1}
\begin{aligned}
\psi(x,y,\alpha) & = \langle x - y, \xi \rangle + \frac{h(\xi)}{2} |y|^2, 
\\
\quad \phi(x,y) & = \phi(x)\phi(y)|\D V_g(x)|^{1/2} |\D V_g(y)|^{1/2} (\D x^{i_1} \ \cdots \ \D x^{i_m})|_{x} \otimes (\D x^{j_1} \ \cdots \ \D x^{j_m})|_{y}.
\end{aligned}
\end{equation}
Now with $\psi$ and $\phi$ as defined in \eqref{psiphi1}, %and noting that $x_0=y_0=0$, 
we see that %\eqref{psymb} becomes
%\shc{I think $\partial_-SM$ should be $\partial_+ SM$ here.}
\begin{equation}\label{eq: Nsym11}
\begin{aligned}
 \langle K_{\mathcal{N}_{m}},e^{-i\psi(x,y,\alpha)} \phi(x,y) \rangle  & =  \int_{\partial_+SM} \left (\int_{\tau_-(v)}^{0} \langle f(\gamma_{v}(t),\alpha),\dot{\gamma}_\tred{v}\rangle  \ \D t \right ) \left (\int_{\tau_-(v)}^{0} \langle g(\gamma_{v}(s),\alpha),\dot{\gamma}_{\tred{v}}\rangle \ \D s \right ) \ |\D \partial_+ SM(v)|\\
 &=\iiint e^{-\I \lambda \wt{\psi}} \phi(\g_{v}(t)) \phi({\g_v}(s))\dot{\g}_{v}^{i_1}(t) \cdots \dot{\g}_{v}^{i_m}(t) \dot{\g}_{v}^{j_1}(s) \cdots \dot{\g}_{v}^{j_m}(s) \ \D t \D s  |\D \partial_+ SM(v)|,
% &=\int_{\partial_- SM} \left (\int_{0}^{\tau(v)} \phi(\gamma_{v}(t))\dot{\g}^{i_1}(t)\cdots \odot{\g}^{m}(t) \s \ \D t \right ) \left (\int_{0}^{\tau(v)} g(\gamma_{v}(s),\alpha) \ \D s \right ) \ |\D \partial_- SM(v)|\
\end{aligned}
\end{equation}
where 
\Beq \label{eq:PDOsp}
\begin{aligned}
\lambda = |\xi|, \quad \xi_0=\frac{\xi}{|\xi|},\quad 
\wt{\psi}(s,t,z,w,\A)= \langle \gamma_{z,w}(t) - \gamma_{z,w}(s),\xi_0 \rangle + \frac{h(\xi_0)}{2} |\gamma_{z,w}(s)|^2. 
\end{aligned}
\Eeq
Here we have used coordinates $(z,w)$ on $\partial_+ SM$, which will be specified \tred{in the next paragraph}, in place of $v$. From now, we will write $h$ for $h(\xi_0)$ to simplify the notation. We apply the method of stationary phase to \eqref{eq: Nsym11} to compute the principal symbol of $\mathcal{N}_m$. From \cite{HU}, we know that if $U$ is sufficiently small, then the only singular points of the phase function occur when $\gamma_v$ passes through $x_0$ orthogonally to $\xi$ and $\gamma_v(t) = \gamma_v(s) = x_0$. In two dimensions there are two such singular points corresponding to the two possible vectors through $x_0$ orthogonal to $\xi$.

Now \tred{we will specify the coordinates $(z,w)$ used in \eqref{eq:PDOsp}}. Suppose that $v_0 \in \partial_+ SM$ is such that $\gamma_{v_0}(t_0) = x_0$ passes through $x_0$ orthogonally to $\xi$. Let us take a coordinate $z$ on $\partial M$ near $\pi(v_0)$ such that $|\D z| = |\D \partial M|$ and a coordinate $w$ on $S_z M$ which is the angle with $v_0$ oriented so that the variation field $\langle D_t \partial_w\gamma_{v_0}(0), \mathcal{I}_{0,t_0} \xi\rangle >0$ where $\mathcal{I}_{0,t_0}$ is parallel translation along $\gamma_{v_0}$. Thus we have a coordinate system on $\partial_+ SM$ near $v_0$ such that $\langle v,\nu \rangle_g |\D z \D w | = |\D \partial_\tred{+} SM|$ where $\nu$ is the \tred{outward} pointing unit normal to the boundary. Assume also that $v_0$ is represented by coordinates $(z_0,w_0)$. %The phase function that appears in \eqref{eq:PDOsp} using these coordinates is then
%\begin{equation}\label{eq:phase22}
%\wt{\psi}(s,t, z,w,\A)= \langle \gamma_{z,w}(t) - \gamma_{z,w}(s),\xi_0 \rangle + \frac{h}{2} |\gamma_{z,w}(s)|^2
%\end{equation}
%where the right hand side is in normal coordinates on $U$.

\tred{Next, we will} calculate the Hessian of \tred{the phase function $\wt{\psi}$ from \eqref{eq:PDOsp}} at the critical point $(t_0,t_0,z_0,w_0)$, but first let us note the following identities which will simplify the calculation:
\begin{equation}\label{eq:ids11}
\gamma_{z_0,w_0}(t_0) = 0, \quad \ddot{\gamma}_{z_0,w_0}(t_0) = 0,\quad \langle \dot{\gamma}_{z_0,w_0}(t_0),\xi_0 \rangle = 0, \quad \langle \partial_w \gamma_{z_0,w_0}(t_0), \dot{\gamma}_{z_0,w_0}(t_0) \rangle = 0.
\end{equation}
The first two identities in \eqref{eq:ids11} follow because $\gamma_{z_0,w_0}$ is a geodesic in normal coordinates through the centre point which is the origin in the coordinate system, the third identity is because $\gamma_{z_0,w_0}$ passes orthogonally to $\xi_0$ and the fourth identity follows from the fact that $\partial_w \gamma_{z_0,w_0}$ is a Jacobi field along $\gamma_{z_0,w_0}$ orthogonal to the direction vector of the geodesic.

In the following formulae, we suppress the dependence of $\gamma$ on $(z,w)$ to simplify the notation. The first derivatives of the phase function in \eqref{eq:PDOsp} are
\begin{align*}
\wt{\psi}_t & = \langle \dot{\gamma}(t),\xi_0 \rangle, & \wt{\psi}_s & =  -\langle \dot {\gamma}(s),\xi_0 \rangle +  h\langle \dot{\gamma}(s),\gamma(s) \rangle\\
 \wt{\psi}_z & =  \langle \partial_z \gamma(t) - \partial_z \gamma(s),\xi_0 \rangle  +  h\langle \partial_z \gamma(s),\gamma(s) \rangle, & \wt{\psi}_w & =  \langle \partial_w \gamma(t) - \partial_w \gamma(s),\xi_0 \rangle  + h\langle \partial_w \gamma(s),\gamma(s) \rangle.
 \end{align*}
 Next we take the second derivatives and evaluate at the critical point $(t_0,t_0,z_0,w_0)$ taking advantage of the identities in \eqref{eq:ids11} to get the Hessian
\begin{equation}\label{eq:Hess0}
H\wt{\psi} = \left (
\begin{array}{cccc}
    0 & 0 & \langle \partial_z \dot{\gamma},\xi_0 \rangle & \langle \partial_w \dot{\gamma},\xi_0 \rangle \\
    0 & h & -\langle \partial_z \dot{\gamma},\xi_0 \rangle + h\langle \dot{\gamma}, \partial_z \gamma \rangle,&-\langle \partial_w \dot{\gamma},\xi_0\rangle \\
   \langle \partial_z \dot{\gamma},\xi_0 \rangle & -\langle \partial_z \dot{\gamma},\xi_0 \rangle + h\langle \dot{\gamma}, \partial_z \gamma \rangle &  h|\partial_z \gamma|^2 & h\langle \partial_z \gamma, \partial_w \gamma \rangle \\
   \langle \partial_w \dot{\gamma},\xi_0 \rangle & -\langle \partial_w \dot{\gamma},\xi_0 \rangle & h\langle \partial_z \gamma, \partial_w \gamma \rangle & h|\partial_w \gamma |^2
\end{array}
\right ).
\end{equation}
Our goal is to calculate the determinant of this Hessian. Using row and column operations in the first step to simplify we obtain, omitting a lengthy calculation,
\begin{align}\label{eq:Hess11}
\mathrm{det}(H\wt{\psi})  & = \det
\begin{pmatrix}
    0 & 0 & \langle \partial_z \dot{\gamma},\xi_0 \rangle & \langle \partial_w \dot{\gamma},\xi_0 \rangle \\
    0 & h &  h\langle \dot{\gamma}, \partial_z \gamma \rangle & 0 
    \\
    \langle \partial_z \dot{\gamma}, \xi_0\rangle &  h\langle \dot{\gamma}, \partial_z \gamma \rangle &  h|\partial_z \gamma|^2 &  h\langle \partial_z \gamma, \partial_w \gamma \rangle
    \\
    \langle \partial_w \dot{\gamma},\xi_0 \rangle & 0 &  h\langle \partial_z \gamma, \partial_w \gamma \rangle &  h|\partial_w \gamma |^2
\end{pmatrix}\\
% \begin{equation}
% =h^2\lb \langle \PD_w\dot{\g},\xi_0\rangle^2\langle \dot{\g},\PD_z \g\rangle^2-\langle \PD_w\dot{\g},\xi_0\rangle^2|\PD_z \g|^2 +2\langle \PD_z\dot{\g},\xi_0\rangle\langle \PD_w\dot{\g},\xi_0\rangle\langle \PD_z\g,\PD_w\g\rangle-\langle \PD_z \dot{\g},\xi_0\rangle^2|\PD_w\g|^2\rb.
% \end{equation}
\label{eq:Hess33}
& =h^2\begin{pmatrix}
\langle \partial_z \dot{\gamma},\xi_0 \rangle & \langle \partial_w \dot{\gamma},\xi_0 \rangle
\end{pmatrix}
\begin{pmatrix}
- |\partial_w \gamma|^2 & \langle \partial_z \gamma, \partial_w \gamma \rangle\\
\langle \partial_z \gamma, \partial_w \gamma \rangle & \langle \partial_z \gamma, \dot{\gamma} \rangle^2 - |\partial_z \gamma|^2
\end{pmatrix}
\begin{pmatrix}
\langle \partial_z \dot{\gamma},\xi_0 \rangle \\ \langle \partial_w \dot{\gamma},\xi_0 \rangle
\end{pmatrix}.
\end{align}
%\tred{In fact, I am getting the following form for the determinant: 
%\[
%\lvert \xi\rvert^2 
%\begin{pmatrix}
%\langle \partial_z \dot{\g},\xi \rangle & \langle \partial_w \dot{\g},\xi \rangle
%\end{pmatrix}
%\begin{pmatrix}
 %   -\lvert \PD_w \g\rvert^2 & \langle \PD_{z}\g, \PD_w \g\rangle \\
  %  \langle \PD_{z}\g, \PD_w \g\rangle &-\lvert \PD_z \g\rvert^2   
%\end{pmatrix}
%\begin{pmatrix}
 %       \langle \partial_z \dot{\gamma},\xi \rangle \\ \langle \partial_w \dot{\gamma},\xi \rangle
  %  \end{pmatrix}.
%\]
%}
Given that the Christoffel symbols vanish at the central point in normal coordinates, this formula can be interpreted invariantly. Indeed, let us express the variational fields $\partial_z \gamma$ and $\partial_w \gamma$ with respect to the parallel frame $\dot{\gamma}$ and $\dot{\gamma}_\perp$ along $\gamma$ as
\begin{equation}\label{eq:J1}
J_z = \partial_z \gamma = a_z \dot{\gamma} + b_z \dot{\gamma}_\perp, \quad J_w = \partial_w \gamma = b_w \dot{\gamma}_\perp.
\end{equation}
Due to the vanishing Christoffel symbols at the critical point, we have
\[
\langle \partial_z \dot{\gamma},\xi_0 \rangle = \langle D_t J_z, \xi_0\rangle, \quad \langle \partial_w \dot{\gamma},\xi_0 \rangle = \langle D_t J_w,\xi_0\rangle
\]
and using $\xi_0 =  \flat_g(\dot{\gamma}_\perp)$ we get
\begin{equation} \label{eq:Jdot1}
\langle \partial_z \dot \gamma, \xi_0 \rangle =  \dot{b}_z, \quad \langle \partial_w \dot{\gamma},\xi_0 \rangle = \dot{b}_w. 
\end{equation}
Putting \eqref{eq:J1} and \eqref{eq:Jdot1} into \eqref{eq:Hess33} we obtain
\[
\begin{split}
\mathrm{det}(H\wt{\psi}) & =  h^2\left (
\begin{array}{cc}
\dot{b}_z & \dot{b}_w
\end{array}
\right )
\left (
\begin{array}{cc}
-b_w^2 & b_zb_w\\
b_zb_w & -b_z^2
\end{array}
\right )
\left (
\begin{array}{c}
\dot{b}_z \\ \dot{b}_w
\end{array}
\right )\\
& = - (\dot{b}_w b_z - \dot{b}_z b_w)^2.
\end{split}
\]
Note that the right hand side is the square of the Wronskian $W\{b_z,b_w\} = \dot{b}_w b_z - \dot{b}_z b_w$. Given that $J_z$ and $J_w$ are both Jacobi fields, we have that for $b = b_z$ or $b = b_w$
\begin{equation}\label{eq:scalcurve1}
\ddot{b} + \kappa b = 0
\end{equation}
where $\kappa$ is the curvature. Therefore the Wronskian is constant along the geodesic. At the boundary, when $t = 0$, the variation with respect to $w$ satisfies
\[
b_w(0) = \langle \partial_w \gamma(0), \dot{\gamma}_\perp \rangle_g = 0, \quad \dot{b}_w(0) = \langle {D}_t\partial_w \gamma(0), \dot{\gamma}_\perp \rangle_g = 1.
\]
Note the second identity follows because we have chosen $w$ to be the angle properly oriented. Therefore
\begin{equation}\label{eq:Wron1}
W\{b_z,b_w\} = b_z(0) = \langle \partial_z \gamma(0), \dot{\gamma}_\perp \rangle_g.
\end{equation}
\tred{Since} $\partial_z \gamma(0)$ is given by the coordinate vector $\frac{\partial}{\partial z}$ and $z$ is an arc length parameter in the boundary $\partial M$, $\partial_z \gamma(0)$ is a unit vector tangent to the boundary and so $b_z(0)$ is the cosine of the angle between \tred{$\dot{\gamma}_\perp$} and the boundary. This is the same as the cosine of the angle between $\dot{\gamma}$ and $\nu$ and so
\begin{equation} \label{eq:detHess1}
\mathrm{det}(H\wt{\psi}) = - h^2\langle \dot{\gamma}(0), \nu \rangle_g^2.
\end{equation}

Next let us compute the signature of $H\wt{\psi}$. Before we do this, let us write the matrix $H\wt{\psi}$ using \eqref{eq:J1}. We have 
\begin{equation*}%\label{eq:Hess01}
H\wt{\psi} = \left (
\begin{array}{cccc}
    0 & 0 &  \dot{b}_z & \dot{b}_w  \\
    0 & h & -\dot{b}_z+ ha_z &-\dot{b}_w  \\
   \dot{b}_z & -\dot{b}_z+ ha_z &  h\lb a_z^2 +b_z^2\rb  & hb_z b_w \\
   \dot{b}_w & -\dot{b}_w & hb_z b_w & hb_w^2
\end{array}
\right ).
\end{equation*}
Since $\det(H\wt{\psi})<0$, the only possibilities for the eigenvalues of $H\wt{\psi}$ are (a) three positive and one negative eigenvalue or (b) three negative and one positive eigenvalue. We will rule out (b) by the calculations below \tred{by computing some coefficients of} the characteristic polynomial of $H\wt{\psi}$ in the variable $\lambda$. We note that the constant coefficient, which corresponds to $\det(H\wt{\psi})$, is negative. Further, the coefficient of $\lambda$ which corresponds to the negative of $\mbox{trace}(H\wt{\psi})$ is negative as well, since the diagonal entries are all non-negative and $h\gg 0$ (to be chosen later). Let us compute the coefficient of $\lambda^2$: 
\Beq\label{eq:lambdacoeff}
\mbox{coeff. of }\lambda^2= -\dot{b}_z^2-\dot{b}_w^2-(ha_z-\dot{b}_z)^2-\dot{b}_w^2-h^2b_z^2 b_w^2+h^2b_w^2(a_z^2+b_z^2)+h^2(a_z^2+b_z^2)+h^2 b_w^2.
\Eeq
Until now, we have not used any specific form for $h$ other than the fact that it is homogeneous of order $1$. We now choose $h$ to be $h=c|\xi|$ for $c\gg0$. Then $h(\xi_0)=c$.
Simplifying \eqref{eq:lambdacoeff}, we obtain
\[
\begin{aligned}
\mbox{coeff. of }\lambda^2&= -2(\dot{b}_z^2+\dot{b}_w^2) +2ha_z\dot{b}_z +h^2(b_w^2a_z^2+b_z^2 +b_w^2)\\
&=-2(\dot{b}_z^2+\dot{b}_w^2) +2ca_z\dot{b}_z +c^2(b_w^2a_z^2+b_z^2 +b_w^2).
\end{aligned}
\]
 Choosing $c$ to be large enough, we can make the coefficient of $\lambda^2$ positive. \tred{Since the coefficient of $\lambda^4$ is 1, the coefficient of $\lambda^3$ is negative and the constant coefficient is negative,} by Descartes' rule of signs, the characteristic polynomial must have 3 positive and one negative root and therefore the signature, $\mbox{sgn}(H\wt{\psi})=2$.
 
\tred{Using the calculations above,} the leading term in the stationary phase expansion of \eqref{eq: Nsym11} corresponding to each critical point is 
%\[
%-\frac{(2\pi)^{2} e^{\I \frac{\pi}{4} \mathrm{sgn}(H\wt{\psi}) }\phi^2(x_0) \dot{\g}^{i_1}_{z_0,w_0}(t_0)\cdots \dot{\g}^{i_m}_{z_0, w_0}(t_0)\dot{\g}^{j_1}_{z_0,w_0}(t_0)\cdots \dot{\g}^{j_m}_{z_0, w_0}(t_0)}{|\xi|^2 \langle \tred{w_0}\dot{\gamma}(0), \nu \rangle_g^2}.
%\]
\[
\tred{\frac{(2\pi)^{2} \I \dot{\g}^{i_1}_{z_0,w_0}(t_0)\cdots \dot{\g}^{i_m}_{z_0, w_0}(t_0)\dot{\g}^{j_1}_{z_0,w_0}(t_0)\cdots \dot{\g}^{j_m}_{z_0, w_0}(t_0)}{c|\xi|^2 \langle w_0, \nu \rangle_g}}.
\]
\tred{Following the calculations of \cite{lan1999operator}, we have $|p_L^* \omega|^{1/2} = c |\xi| |\D x \D \xi|^{1/2}$ (see also \eqref{eq:pLwcalc} and preceding formulae). Multiplying this by the previous formula gives the contribution of one of the two critical points corresponding to vectors orthogonal to $\xi$ at $x_0$ to the principal symbol. Combining the two contributions,} the principal symbol $\sigma(x_0,\xi, x_0, \xi)$ is given by 
%\Beq
%\begin{aligned}
%   \sigma(x_0,\xi, x_0,\xi)& = -\frac{(2\pi)^{2} e^{\I \frac{\pi}{4} \mathrm{sgn}(H\wt{\psi}) }\phi^2(x_0) \dot{\g}^{i_1}_{z_0,w_0}(t_0)\cdots \dot{\g}^{i_m}_{z_0, w_0}(t_0)\dot{\g}^{j_1}_{z_0,w_0}(t_0)\cdots \dot{\g}^{j_m}_{z_0, w_0}(t_0)}{|\xi|^2 \langle  \tred{w_0}\dot{\gamma}(0), \nu (z_0) \rangle_g}\\
 %  &-\frac{(2\pi)^{2} e^{\I \frac{\pi}{4} \mathrm{sgn}(H\wt{\psi}) }\phi^2(x_0) \dot{\g}^{i_1}_{z_1,w_1}(t_1)\cdots \dot{\g}^{i_m}_{z_1, w_1}(t_1)\dot{\g}^{j_1}_{z_1, w_1}(t_1)\cdots \dot{\g}^{j_m}_{z_1, w_1}(t_1)}{|\xi|^2 \langle \dot{\gamma}(0) \tred{w_1}, \nu (z_1) \rangle_g}.
%\end{aligned}
%\Eeq
\Beq
\begin{aligned} \label{PS:gamma}
   \tred{\sigma^{i_1 \cdots i_m j_1 \cdots j_m}
   (x_0,\xi, x_0,\xi)} & \tred{ = \frac{(2\pi)^{2}\I \dot{\g}^{i_1}_{z_0,w_0}(t_0)\cdots \dot{\g}^{i_m}_{z_0, w_0}(t_0)\dot{\g}^{j_1}_{z_0,w_0}(t_0)\cdots \dot{\g}^{j_m}_{z_0, w_0}(t_0)}{|\xi| \langle  w_0, \nu (z_0) \rangle_g}|\D x \D \xi|^{1/2}}\\
   &\tred{\hskip1cm +\frac{(2\pi)^{2} \I \dot{\g}^{i_1}_{z_1,w_1}(t_1)\cdots \dot{\g}^{i_m}_{z_1, w_1}(t_1)\dot{\g}^{j_1}_{z_1, w_1}(t_1)\cdots \dot{\g}^{j_m}_{z_1, w_1}(t_1)}{|\xi| \langle w_1, \nu (z_1) \rangle_g}|\D x \D \xi|^{1/2}}.
\end{aligned}
\Eeq
%\shc{Should this be a half density?}
% The expression for the principal symbol we have obtained depends on the choice of the cutoff functions $\phi$ and the phase function $\wt{\psi}$. To make it independent of the choice of these functions,  we divide the principal symbol obtained by the principal symbol of the identity operator. The Schwartz kernel of the identity operator $\delta_{\Delta}(x,y)$, where $\Delta$ is the diagonal set in $U\times U$, where $U$ is an open subset of $\Rb^n$ obtained using the normal coordinates centered at $x_0$ already used above. We compute the principal symbol of the identity operator using the same choice of cutoff functions and the phase function $\wt{\psi}$. Applying the method of stationary phase, the principal symbol of the identity operator is 
% \[
% \frac{(2\pi\I)\phi^{2}(x_0)}{|\xi|}.
% \]
% Dividing $\sigma$ by this expression, we get,%\shc{Is $\sigma$ defined?} 
% %\vkc{Thanks for pointing this out. I will define it.}
% \Beq
% \begin{aligned}
%    \sigma(x_0,\xi, x_0,\xi)& = -\frac{2\pi\dot{\g}^{i_1}_{z_0,w_0}(t_0)\cdots \dot{\g}^{i_m}_{z_0, w_0}(t_0)\dot{\g}^{j_1}_{z_0,w_0}(t_0)\cdots \dot{\g}^{j_m}_{z_0, w_0}(t_0)}{|\xi| \langle \dot{\gamma}(0), \nu (z_0) \rangle_g}\\
%    &-\frac{2\pi   \dot{\g}^{i_1}_{z_1,w_1}(t_1)\cdots \dot{\g}^{i_m}_{z_1, w_1}(t_1)\dot{\g}^{j_1}_{z_1, w_1}(t_1)\cdots \dot{\g}^{j_m}_{z_1, w_1}(t_1)}{|\xi| \langle \dot{\gamma}(0), \nu (z_1) \rangle_g}.
% \end{aligned}
% \Eeq
We can write the principal symbol above in terms of only $(x_0,\xi)$ as follows. We let $\xi^{\perp}$ to be the vector obtained by $90^{\circ}$ anticlockwise rotation of the covector $\xi$. \tred{This is written in an coordinate invariant way as
\[
\xi^\perp = \sharp_g(\ast_g \xi).
\]
}
%\shc{I've reformulated this in terms of the Hodge star and added a note in the preliminaries defining the Hodge star. I know we discussed this and I'm not sure what was decided; we can change it back if you prefer, but I like this version because it is coordinate invariant.}
%That is, 
% \[
% \xi^{\perp} = \sqrt{\det g}\,  g^{-1} J g^{-1} \xi, \mbox{ where } J=\begin{pmatrix} 0 & -1\\ 1 & 0\end{pmatrix}.
% \]
%\[
%\xi^{\perp} =\sqrt{\det g}\,\sharp_{g}
%\]
Let $\xi^{\perp}_0$ be the unit vector $\xi^{\perp}_0 = \frac{\xi^{\perp}}{|\xi^{\perp}|_{g}}$. 
The geodesic starting at $(x_0,\xi^{\perp}_0)$ intersects the boundary at finite times $\tau_{\pm}(x_0,\xi^{\perp}_0)$. Letting $\g_{x_0,\xi_0^{\perp}}(t)$ be  this geodesic, we then have 
\Beq\label{z0z1-Expressions}
\g_{x_0,\xi^{\perp}_0}(\tau_{+}(x_0,\xi^{\perp}_0))=z_0, \quad \g_{x_0,\xi^{\perp}_0}(\tau_{-}(x_0,\xi^{\perp}_0))=z_1.
\Eeq
We also have 
\Beq\label{w0w1-Expressions}
w_0=-\dot{\g}_{x_0,\xi^{\perp}_0}(\tau_{+}(x_0,\xi_0^{\perp})) \quad \mbox{and}\quad w_1=\dot{\g}_{x_0,\xi^{\perp}_0}(\tau_{-}(x_0,\xi_0^{\perp})).
\Eeq
\tred{Using these notations, we see that \eqref{PS:gamma} becomes \eqref{PS:PsiDO}. This completes the proof of \eqref{PS:PsiDO} and shows how it depends on $(x_0,\xi)$.} 

% Then finally, 
% \Beq
% \begin{aligned}
%    \sigma^{i_1 \cdots i_m j_1 \cdots j_m}(x_0,\xi, x_0,\xi)& = -\frac{(2\pi )\lb \xi^{\perp}\rb ^{i_1} \cdots \lb \xi^{\perp}\rb^{i_m}\lb \xi^{\perp}\rb^{j_1}\cdots \lb \xi^{\perp}\rb^{j_m}}{|\xi|}\lb \frac{1}{ \langle w_0, \nu (z_0) \rangle_g}+ \frac{1}{ \langle w_1, \nu (z_1) \rangle_g}\rb, 
% \end{aligned}
% \Eeq
% with $z_0$ and $z_1$ as above.

Next, let us show that the principal symbol $\sigma$ computed above is elliptic on solenoidal tensor fields. By this we mean, the following: If $f$ is a symmetric tensor field such that $\sigma(x_0, \xi, x_0,\xi)f=0$ and $\xi^{i} f_{i_1 \cdots i_{m-1} i}(\xi)=0$, then $f\equiv 0$. Multiplying $\sigma f$ by $f$ and summing over all the indices, we get 
\[
\lb \xi^{\perp}\rb^{i_1} \cdots \lb \xi^{\perp}\rb ^{i_m} f_{i_1 \cdots i_m}=0.
\]
Furthermore from $\xi^{i} f_{i_1 \cdots i_{m-1} i}(\xi)=0$, we get, 
\[
\lb \xi^{\perp}\rb^{i_1} \cdots \lb \xi^{\perp}\rb^{i_{m-1}} \xi^{i_m} f_{i_1 \cdots i_m}(\xi)=0.
\]
The above two equalities combined with the fact that $\xi$ and $\xi^{\perp}$ form a basis for the tangent space at the $x_0$ gives that
\[
\lb \xi^{\perp}\rb^{i_1} \cdots \lb \xi^{\perp}\rb^{i_{m-1}} f_{i_1 \cdots i_{m-1} i}(\xi)=0 \mbox{ for } i = 1,2.
\]
Let us fix $i$ and name the $m-1$ symmetric tensor field $\wt{f}$ with $i$ fixed as
\[
\wt{f}_{i_1 \cdots i_{m-1}}=f_{i_1 \cdots i_{m-1}i}.
\]
Next we have, with the same $i$ fixed above, 
\[
\lb \xi^{\perp}\rb^{i_1} \cdots \lb \xi^{\perp}\rb^{i_{m-2}} \xi^{i_{m-1}} f_{i_1 \cdots i_{m-1} i}(\xi)=0.
\]
Proceeding as before, we arrive at 
\[
\lb \xi^{\perp}\rb^{i_1} \cdots \lb \xi^{\perp}\rb^{i_{m-2}} f_{i_1 \cdots i_{m-2} i j}(\xi)=0 \mbox{ for } i,j = 1,2.
\]
Inductively proceeding, we arrive at $f_{i_1 \cdots i_{m}}(\xi)=0$ for any fixed index $i_1, \cdots, i_m$. \tred{This establishes the assertion that $\sigma$ is elliptic on solenoidal tensor fields and therefore completes the proof of Theorem \ref{Thm:PS:PsiDO}.}
\epr

Next we compute the principal symbol of the FIO components of the operator $\Nc_{\tred{m}}$. %This is given by the following theorem
%\section{Principal symbol of the FIO components of $I^{*}I$} \label{sec:FIO}

\begin{theorem}\label{Thm:PS:FIO} The principal symbol 
$\wt{\sigma}_{\tred{m}}$ of the operator of each FIO component of \eqref{eq:Ndecomp} is
\Beq\label{PS:FIO}
\begin{aligned}
   \wt{\sigma}_{\tred{m}}^{i_1 \cdots i_m j_1 \cdots j_m}(x_0,\xi, y_0,\eta) &= \frac{(2\pi\tred{)^2} \I\lb \xi^{\perp}_{\tred{0}}\rb ^{i_1} \cdots \lb \xi^{\perp}_{\tred{0}}\rb^{i_m}\lb \eta^{\perp}_{\tred{0}}\rb^{j_1}\cdots \lb \eta^{\perp}_{\tred{0}}\rb^{j_m}}{|\xi|_\tred{g}}\\
   &\hskip4cm \times \lb \frac{1}{\langle w_0, \nu\tred{(z_0)}\rangle_{g}f(t_0,s_0)} + \frac{1}{\langle w_1, \nu\tred{(z_1)}\rangle_{g}f(t_1,s_1)}\rb   |\D y \D \xi|^{1/2}.
\end{aligned}
\Eeq
with $z_0, z_1$ and $w_0, w_1$ given in \eqref{z0z1-Expressions} and \eqref{w0w1-Expressions}, respectively, \tred{the values $t_0$ and $s_0$ are the geodesic distances between the points $z_0$ and $x_0$ and $z_0$ and $y_0$ respectively measured along the geodesic starting at $z_0$ in the direction $w_0$, and $t_1$ and $s_1$ are the same in the opposite direction. The covectors $\xi_0$ and $\eta_0$ are defined in \eqref{xieta}}. The vector field $\nu$ is the unit \tred{outer} normal at the boundary, and $f(t,s)$ is the solution to \eqref{eq:scalcurve2} below with \eqref{fConditions}.
\end{theorem} 
\bpr
Let $\alpha = (x_0,\xi;y_0,\eta)$ be in $\tred{\Cc}_J$ defined by \eqref{Lj}. Then there is a geodesic $\g:[0,1]\to M$ passing through $x_0$ and $y_0$, say, $\g(t_0)=x_0$ and $\g(s_0)=y_0$ with $0<t_0<s_0<1$, and there is a Jacobi field $J$ along $\g$ vanishing at $t_0$ and $s_0$ such that $\flat_\tred{g}(D_tJ(t_0)) = \xi$ and $\flat_\tred{g}(D_tJ(s_0)) = \eta$. We choose normal coordinates \tred{in neighbourhoods $U_{x_0}$ and $U_{y_0}$ of} $x_0$ and $y_0$
%\vkc{Is this (being parallel to the first coordinate axes) really required?}
%\shc{It doesn't appear so, I've removed that.}
\tred{and take any $\phi_{x_0} \in C^\infty_c(U_x)$ and $\phi_{y_0} \in C^\infty_c(U_y)$ such that $\phi_{x_0}(x_0) = 1$ and $\phi_{y_0}(y_0) = 1$. Using the normal coordinates, let us define two symmetric $m$-tensor fields $f$ and $g$ in $C_c^\infty(U_{x_0}; S^{m}\tau'_{U_{x_0}})$ and $C_c^\infty(U_{y_0}; S^{m}\tau'_{U_{y_0}})$ respectively by
\[
f(x,\alpha) =  \phi(x) e^{-i \langle x, \xi \rangle} \D x^{i_1}\ \cdots \D x^{i_m}|_{x}, \quad g(y,\alpha) = \phi(y) e^{i\left ( \langle y, \eta \rangle - \frac{h(\xi,\eta)}{2}|y|^2 \right )}\lb \D y^{j_1} \ \cdots \ \D y^{j_m}\rb |_{y}
\]
The choice of function $h$ will be made precise later, but is homogeneous of order $1$ in $\xi$ and $\eta$. Again using the two sets of coordinates we define}
\begin{equation}
\begin{aligned}\label{eq:FIOphase1}
    \tred{\quad \psi (x,y,\A)} & \tred{= \langle x,\xi\rangle -\langle y,\eta\rangle + \frac{1}{2}|y|^2 h(\xi,\eta),}\\\tred{
    \phi(x,y)} & 
    \tred{= \phi(x)\phi(y)|\D V_g(x)|^{1/2} |\D V_g(y)|^{1/2} (\D x^{i_1} \ \cdots \ \D x^{i_m})|_{x} \otimes (\D y^{j_1} \ \cdots \ \D y^{j_m})|_{y}.}
\end{aligned}
\end{equation}
% let us consider the phase function 
% \begin{equation}\label{eq:conjphase}
% \psi (x,y,\A) = (x_0,\xi;y_0,\eta)=\langle x,\xi\rangle -\langle y,\eta\rangle + \frac{1}{2}|y|^2 h(\xi,\eta),
% \end{equation}
% which is defined for $x$ and $y$ within the normal coordinate neighbourhoods of $x_0$ and $y_0$. The choice of function $h$ will be made precise later, but is homogeneous of order $1$ in $\xi$ and $\eta$. \tred{Also, let $\phi(x,y)$ be defined as in \eqref{psiphi1} where we stress that the tensor fields are defined in the coordinates near $x_0$ and $y_0$ respectively.} 
\tred{To determine the principal symbol, w}e consider
\begin{equation}\label{eq:Nsym1}
\begin{split}
e^{i \psi(x_0,y_0,\alpha)} \langle K_{\mathcal{N}_\tred{m}},e^{-i\psi(x,y,\alpha)} \phi(x,y) \rangle  & =\\ 
& \hskip-4cm e^{i \psi(x_0,y_0,\alpha)} \int_{\partial_+SM} \left (\int_{\tred{\tau_-(v)}}^{\tred{0}} \tred{\langle} f(\gamma_{v}(t),\alpha)\tred{,\dot{\gamma}_v\rangle} \ \D t \right ) \left (\int_{\tred{\tau_-(v)}}^{\tred{0}} \tred{\langle} g(\gamma_{v}(s),\alpha)\tred{,\dot{\gamma}_v\rangle} \ \D s \right ) \ |\D \partial_- SM(v)|.
\end{split}
\end{equation}
As in the \tred{proof of Theorem \ref{Thm:PS:PsiDO}}, we will apply the method of stationary phase to \eqref{eq:Nsym1}. % but with the phase function defined by \eqref{eq:FIOphase1} and cut-off functions non-zero in the normal coordinate neighbourhoods.
\tred{Following the same method,} analogous to \eqref{eq:PDOsp} \tred{we have} 
\begin{align}\label{xieta}
    \tred{\lambda = |\xi|, \quad \xi_0 = \frac{\xi}{|\xi|}, \quad \eta_0 = \frac{\eta}{|\xi|}, \quad 
    \wt{\psi}(s,t,z,w,\alpha)=\langle \g_{z,w}(t),\xi_0\rangle - \langle \g_{z,w}(s),\eta_0\rangle + \frac{1}{2}|\g_{z,w}(s)|^2 h(\xi_0,\eta_0).}
\end{align}
As before, $(t_0, s_0,z_0, w_0)$ is a critical point of this phase function \tred{$\wt{\psi}$} and we will calculate the Hessian of the phase function at this critical point. To simplify the notation, we will drop the subscript\tred{s} $z$ and $w$ from $\g$ but the reader should bear in mind that $\g$ depends on these two variables.

The first derivatives of the phase function are
\begin{align*}
& \tred{\wt{\psi}}_t=\langle \dot{\g}(t),\xi_\tred{0}\rangle, \quad \tred{\wt{\psi}}_{s}=-\langle \dot{\g}\tred{(s)},\eta_\tred{0}\rangle + h\langle \dot{\g}(s), \g(s)\rangle,\\
&\tred{\wt{\psi}}_z=\langle \PD_z \g(t),\xi_\tred{0}\rangle -\langle \PD_z \g(s),\eta_\tred{0}\rangle + h\langle \PD_z \g(s), \g(s)\rangle, \\
&\tred{\wt{\psi}_w}=\langle \PD_w \g(t),\xi_\tred{0}\rangle -\langle \PD_w \g(s),\eta_\tred{0}\rangle+ h\langle \PD_w \g(s), \g(s)\rangle.
\end{align*}
Next let us compute the Hessian matrix $H\tred{\wt{\psi}}$ at the critical point (that is at $t=t_0, s=s_0, z=z_0, w=w_0$)\tred{. For this we note that the identities \eqref{eq:ids11} still hold as well as
\[
\g(s_0) = 0, \quad \ddot{\g}(s_0) = 0, \quad \langle \dot{\g}_{z_0,w_0}(s_0),\eta_0 \rangle = 0, \quad \langle \partial_w \g_{z_0,w_0}(s_0),\dot{\g}_{z_0,w_0}(s_0) \rangle = 0.
\]
To make the following formulae manageable, we have suppressed the dependence on $\g$ on $s_0$ or $t_0$. It is possible to tell at which point $\g$ should be evaluated by context. Note also that $h$ is evaluated at $(\xi_0,\eta_0)$. Calculation shows}
\small
\Beq\label{Hpsimatrix}
\begin{split}
&H \tred{\wt{\psi}}=
\left (\begin{matrix}
0 & 0 \\
0& h \\
\langle\PD_z \dot{\g},\xi_\tred{0}\rangle & -\langle \PD_z \dot{\g},\eta_\tred{0}\rangle + h\langle \PD_z \g, \dot{\g}\rangle \\
\langle \PD_{w}\dot{\g},\xi_\tred{0}\rangle & -\langle \PD_w \dot{\g} ,\eta_\tred{0}\rangle
\end{matrix}
\right . \\
&\hskip1.5in \left .\begin{matrix}
\langle \PD_z \dot{\g},\xi_\tred{0}\rangle & \langle \PD_{w}\dot{\g},\xi_\tred{0}\rangle \\
-\langle \PD_z \dot{\g},\eta_\tred{0}\rangle + h\langle \PD_z \g, \dot{\g}\rangle & -\langle \PD_w \dot{\g} ,\eta_\tred{0}\rangle \\
\langle \PD_{zz} \g,\xi_\tred{0}\rangle -\langle \PD_{zz} \g,\eta_\tred{0}\rangle + h|\PD_z \g|^2 & \langle \PD_{zw} \g,\xi_\tred{0}\rangle - \langle \PD_{zw} \g,\eta_\tred{0}\rangle + h \langle \PD_z \g,\PD_w \g\rangle \\
\langle \PD_{zw} \g,\xi_\tred{0}\rangle - \langle \PD_{zw} \g,\eta_\tred{0}\rangle + h \langle \PD_z \g,\PD_w \g\rangle &\langle \PD_{ww} \g,\xi_\tred{0}\rangle - \langle \PD_{ww} \g,\eta_\tred{0}\rangle + h |\PD_w \g|^2
\end{matrix}
\right ).
\end{split}
\Eeq
\normalsize
To simplify the calculation of the determinant of $H\tred{\wt{\psi}}$, it is convenient to express the entries in the Hessian abstractly so that
\Beq\label{Hpsimatrixabstract}
H\tred{\wt{\psi}} = 
\begin{pmatrix}
    0 & 0 & a & b\\
    0 & h & c & d\\
    a & c & e & f\\
    b & d & f & g
\end{pmatrix}.
\Eeq
Expanding the determinant gives
\[
\mathrm{det}(H\tred{\wt{\psi}}) = h 
\begin{pmatrix}
    a & b
\end{pmatrix}
\begin{pmatrix}
    -g & f\\ f & -e
\end{pmatrix}
\begin{pmatrix}
    a \\ b
\end{pmatrix}
+ (ad-bc)^2.
\]
Substituting in the values of the entries of the matrix, and simplifying slightly, we have 
\small
\begin{equation}\label{eq:detHconj1}
\begin{split}
& \mathrm{det}(H\tred{\wt{\psi}}) = \\ 
&\hskip1.5cm h\begin{pmatrix}
    \langle \partial_z\dot{\g},\xi_\tred{0} \rangle & \langle \partial_w\dot{\gamma},\xi_\tred{0} \rangle
\end{pmatrix}
\begin{pmatrix}
    -\langle \partial_{ww}\gamma,\xi_\tred{0} \rangle + \langle \partial_{ww}\gamma,\eta_\tred{0} \rangle & \langle \partial_{zw}\gamma,\xi_\tred{0} \rangle - \langle \partial_{zw}\gamma,\eta_\tred{0} \rangle \\ \langle \partial_{zw}\gamma,\xi_\tred{0} \rangle - \langle \partial_{zw}\gamma,\eta_\tred{0} \rangle & -\langle \partial_{zz}\gamma,\xi_\tred{0} \rangle + \langle \partial_{zz}\gamma,\eta_\tred{0} \rangle
\end{pmatrix}
\begin{pmatrix}
    \langle \partial_z\dot{\gamma},\xi_\tred{0} \rangle \\ \langle \partial_w\dot{\gamma},\xi_\tred{0} \rangle
\end{pmatrix}\\
&\hskip1.7cm + h^2 \begin{pmatrix}
    \langle \partial_z\dot{\gamma},\xi_\tred{0} \rangle & \langle \partial_w\dot{\gamma},\xi_\tred{0} \rangle
\end{pmatrix}
\begin{pmatrix}
    -|\partial_w\gamma|^2 & \langle \partial_z \gamma,\partial_w \gamma \rangle \\ \langle \partial_z \gamma,\partial_w \gamma \rangle &  \langle \partial_z \gamma, \dot{\gamma} \rangle^2 -|\partial_z\gamma|^2
\end{pmatrix}
\begin{pmatrix}
    \langle \partial_z\dot{\gamma},\xi_\tred{0} \rangle \\ \langle \partial_w\dot{\gamma},\xi_\tred{0} \rangle
\end{pmatrix}
\\& \hskip1.7cm + \lb \langle \PD_z \dot{\g},\xi_\tred{0}\rangle \langle \PD_w\dot{\g},\eta_\tred{0}\rangle -\langle \PD_w \dot{\g},\xi_\tred{0}\rangle \langle \PD_{z}\dot{\g}, \eta_\tred{0}\rangle\rb^2\\
&\hskip1.7cm+ 2h\Bigg{\{}\langle \PD_w\dot{\g},\xi_\tred{0}\rangle \langle \PD_z \dot{\g},\xi_\tred{0}\rangle \langle \PD_z \g,\dot{\g}\rangle \langle \PD_w \dot{\g},\eta_\tred{0}\rangle -\langle \PD_w \dot{\g},\xi_\tred{0}\rangle^2 \langle \PD_z \dot{\g}, \eta_\tred{0}\rangle \langle \PD_z \g,\dot{\g}\rangle\Bigg{\}}.
\end{split}
\end{equation}
\normalsize
According to \cite[Lemma 4.1.1]{D}, \tred{the density}
\begin{equation} \label{eq:invd}
    \tred{\frac{|p_L^* \omega|}{|\det(H\wt{\psi})|}}
\end{equation}
\tred{does not depend on the choice of $h$. We will use this fact to evaluate the principal symbol by taking the limit of this density as $h(\xi_0,\eta_0) \rightarrow \infty$.} To calculate $p_L^*\omega$, we must introduce coordinates $\{z^j\}_{j=1}^4$ on 
\[
\tred{\Lambda_J = \{(x,\xi;y,\eta) \ : \ (x,\xi;y,-\eta)\in \Cc_J \},.}
\]
Indeed, following the calculations in \cite{lan1999operator}, suppose that the coordinate vectors with respect to these coordinates are given in the natural coordinates on $T_\alpha T^* (M\times M)$ induced by the normal coordinates on $M$ near $x_0$ and $y_0$ as columns of
\[
\begin{pmatrix}
D \pi_b\\
D \pi_f
\end{pmatrix}
\]
where the upper part refers to the tangent vectors to the base while the lower part are tangent vectors to the fiber. By the calculations on \cite[Page 66]{lan1999operator}, we have
\[
|p_L^* \omega| = \left |\mathrm{det}\left (D \pi_f - 
\begin{pmatrix}
0_{2} & 0_2\\
0_2 & h\tred{(\xi,\eta)} \mathrm{I}_2
\end{pmatrix} D \pi_b \right ) \right | |dz|.
\]
If we choose the coordinates on $\Lambda_J$ to be $z^{1,2} = \xi_{1,2}$ and $z_{3,4} = y_{3,4}$ then the previous formula gives
\begin{equation}\label{eq:pLwcalc}
|p_L^* \omega| = (h\tred{(\xi,\eta)}^2 + h\tred{(\xi,\eta)} a + b) |\mathrm{d} \xi \mathrm{d} y|
\end{equation}
where $a$ and $b$ are functions which do not depend on $h$. Finally, let $h(\xi,\eta) = c |\xi|$ for a positive constant $c$. Then since \tred{\eqref{eq:invd}} does not depend on $c$, and $\xi \neq 0$, we get
\small
\begin{equation}\label{eq:vol2}
\begin{split}
& \frac{|p_L^*\omega|}{|\mathrm{det}(H\wt{\psi})|} = \lim_{c \rightarrow \infty} \frac{|p_L^*\omega|}{|\mathrm{det}(H\wt{\psi})|} = \\
&\hskip2.5cm 
\tred{|\xi|^{2}} \left [\begin{pmatrix}
    \langle \partial_z\dot{\gamma},\xi_{\tred{0}} \rangle & \langle \partial_w\dot{\gamma},\xi_{\tred{0}} \rangle
\end{pmatrix}
\begin{pmatrix}
    -|\partial_w\gamma|^2 & \langle \partial_z \gamma,\partial_w \gamma \rangle \\ \langle \partial_z \gamma,\partial_w \gamma \rangle &  \langle \partial_z \gamma, \dot{\gamma} \rangle^2 -|\partial_z\gamma|^2
\end{pmatrix}
\begin{pmatrix}
    \langle \partial_z\dot{\gamma},\xi_{\tred{0}} \rangle \\ \langle \partial_w\dot{\gamma},\xi_{\tred{0}} \rangle
\end{pmatrix}\right ]^{-1} |\mathrm{d} \xi \mathrm{d} y|.
\end{split}
\end{equation}
\normalsize
Now we use \eqref{eq:J1}, \eqref{eq:Jdot1} and the formula in between these two to reveal the invariant structure of \eqref{eq:vol2}. Indeed,
\[
\begin{split}
& \begin{pmatrix}
    \langle \partial_z\dot{\gamma},\xi_{\tred{0}} \rangle & \langle \partial_w\dot{\gamma},\xi_{\tred{0}} \rangle
\end{pmatrix}
\begin{pmatrix}
    -|\partial_w\gamma|^2 & \langle \partial_z \gamma,\partial_w \gamma \rangle \\ \langle \partial_z \gamma,\partial_w \gamma \rangle & \langle \partial_z \gamma, \dot{\gamma} \rangle^2 -|\partial_z\gamma|^2
\end{pmatrix}
\begin{pmatrix}
    \langle \partial_z\dot{\gamma},\xi_{\tred{0}} \rangle \\ \langle \partial_w\dot{\gamma},\xi_{\tred{0}} \rangle
\end{pmatrix} \\
&\hskip5cm  = \begin{pmatrix}
    \dot{b}_z(t_{\tred{0}}) & \dot{b}_w(t_{\tred{0}})
\end{pmatrix}
\begin{pmatrix}
    -b_w(s_{\tred{0}})^2 & b_z(s_{\tred{0}})b_w(s_{\tred{0}}) \\ b_z(s_{\tred{0}})b_w(s_{\tred{0}}) & -b_z(s_{\tred{0}})^2
\end{pmatrix}
\begin{pmatrix}
    \dot{b}_z(t_{\tred{0}}) \\ \dot{b}_w(t_{\tred{0}})
\end{pmatrix}\\
&\hskip5cm = - \tred{(b_z(s_{0})\dot{b}_w(t_{0})-\dot{b}_z(t_{0}) b_w(s_{0}))^2}.
\end{split}
\]
Let us label the term in parentheses in the final line above as
\[
B(t,s) = \tred{b_z(s)\dot{b}_w(t)-\dot{b}_z(t) b_w(s)}.
\]
Note that when $s = t$, $B(t,t) = W\{b_z,b_w\} =\tred{\langle \dot{\g}(0), \nu\rangle_g}$ by \eqref{eq:Wron1} \tred{and the argument leading to \eqref{eq:detHess1}}. Also, $\partial_s B(t,t) = 0$. Finally, using \eqref{eq:scalcurve1} we see
\begin{equation}\label{eq:scalcurve2}
\partial^2_s B(t,s) + \kappa(s) B(t,s) = 0.
\end{equation}
Let $f(t,s)$ be the solution of \eqref{eq:scalcurve2} with
\Beq\label{fConditions}
f(t,t) = 1, \quad \PD_s f(t,s)|_{s=t} = 0.
\Eeq
Then,
\[
B(t,s) =  \tred{\langle \dot{\g}(0), \nu\rangle_g} f(t,s)
\]
and we have now shown, recalling \eqref{eq:vol2}, that
\[
\tred{\frac{|p_L^*\omega|}{|\mathrm{det}(H\wt{\psi})|}} =  \frac{\tred{|\xi|^2_{g}}}{\tred{\langle \dot{\g}(0), \nu\rangle_g}^2 f(t_{\tred{0}},s_{\tred{0}})^2}|\mathrm{d}\xi \mathrm{d} y|.
\]
This shows the dependence on the initial angle of the geodesic with the boundary and the curvature along the geodesic between the pair of conjugate points.

%\subsection{Signature of the Hessian}

We next compute the signature of the Hessian. \tred{The} calculation \tred{is} similar to what we did for the $\Psi$DO part. The coefficient of $\lambda^2$ in the characteristic polynomial of the matrix \eqref{Hpsimatrix} in terms of 
\tred{\eqref{Hpsimatrixabstract}} is 
\[
\mbox{coefficient of }\lambda^2: -a^2-b^2-c^2-d^2-f^2+eg+eh+gh.
\]
From \eqref{Hpsimatrix}, the terms involving $h^2$ in the coefficient of $\lambda^2$ are
\[
-\langle \PD_z \g,\dot{\g}\rangle^2-\langle \PD_z \g,\PD_w \g\rangle^2 + |\PD_z \g|^2|\PD_w \g|^2 + |\PD_z \g|^2 + |\PD_w \g|^2.
\]
We note that using Cauchy-Schwartz inequality, the sum above is non-negative, and in fact, it is strictly positive. Hence as before, choosing $\tred{c}$ to be large enough, we can make the coefficient of $\lambda^2$ positive. Again by Descartes' rule of signs, the characteristic polynomial must have 3 positive and one negative root and therefore the signature of $\mbox{sgn}(H\wt{\psi})=2$ in this case as well.

%\subsection{Principal symbol in the conjugate points case}

Finally, repeating the arguments from the section on $\Psi$DO case in this set up, we obtain the principal symbol to be \tred{given by \eqref{PS:FIO}.}

\epr

\section{Constructing a parametrix and cancellation of singularities}\label{sec:cancel}

Let us write the decomposition \eqref{eq:Ndecomp} as
\begin{equation}\label{eq:Ndecomp1}
\mathcal{N} \equiv \Psi + A
\end{equation}
where $A = \sum_{j=1}^k A_j$ is a sum of FIOs. In fact, we may assume a little bit more about the decomposition \eqref{eq:Ndecomp1} as we record in the next lemma.

\begin{lemma}\label{lem:WLOG}
It is possible to choose $\Psi$ and $A$ in \eqref{eq:Ndecomp1} which are symmetric and map solenoidal to solenoidal and potential to potential components. Furthermore, they can both be chosen to vanish on the potential part of the solenoidal decomposition.
\end{lemma}

\begin{proof}
Symmetrising both sides of \eqref{eq:Ndecomp} and using the fact that $\mathcal{N}$ is symmetric, we see that $\Psi$ and $A$ can be made symmetric. Since the canonical relation of $A$, given by $\eqref{Lj}$, is symmetric the symmetrisation of $A$ remains a sum of FIOs with the same canonical relations, which are themselves symmetric. Then composing on both sides with the projection $\pi_S$ onto the solenoidal part, which is a pseudodifferential operator, and using the fact that $\pi_S \mathcal{N} \pi_S \equiv \mathcal{N}$ we see $\Psi$ and $A$ can be chosen to respect the solenoidal decomposition. Finally, since $\mathcal{N}$ vanishes on the potential part of the solenoidal decomposition, $\Psi$ and $A$ can be chosen to do so as well by considering the projection of \eqref{eq:Ndecomp1} onto the potential part.
\end{proof}

As shown in the concluding part of \S \ref{sec:PDO}, $\Psi$ is elliptic on solenoidal tensors and specifically the operator
\begin{equation} \label{eq:C}
C = \Psi + \D \Lambda \delta
\end{equation}
is elliptic where $\delta$ is the divergence defined with respect to the metric $g$ and $\Lambda$ is a scalar symmetric elliptic $\Psi$DO of order $-3$. To avoid problems with the boundary, we will consider a pseudodifferential parametrix for $C$ within an open manifold $\widetilde{M}$ compactly contained in the interior of $M$. For example, we can take $\widetilde{M}$ to be the set which is at least a distance $\epsilon$ from $\partial M$. Given such an $\widetilde{M}$, let $\varphi \in C_c^\infty(M)$ be equal to one on $\widetilde{M}$. In the following a {\it parametrix} of $C$ is an operator $T$ such that
\begin{equation}\label{eq:parametrix}
T C = \varphi (\mathrm{Id} + K)
\end{equation}
where $K$ is an operator on $M$ with smooth Schwartz kernel. We also will use the notation
\begin{equation} \label{eq:sim}
T C \equiv \mathrm{Id}
\end{equation}
to denote \eqref{eq:parametrix}. This notation and the term parametrix will also be used for other operators. Similar to \eqref{eq:sim}, we will write $f \equiv g$ if $f$ and $g$ are equal modulo a smooth function on $M$.

Next we show how a parametrix for $C$ behaves when acting on $\mathcal{N}$. The proof follows the method used in \cite{chappa2006characterization} for the case without conjugate points.

\begin{lemma}\label{lem:para}
There exists a symmetric parameterix $T$ for $C$ which respects the solenoidal decomposition. Also, we have
\begin{equation} \label{eq:Nmpara}
T\mathcal{N} \equiv \pi_S + TA.
\end{equation}
\end{lemma}

\begin{proof}
First, note that $C$ is symmetric and respects the solenoidal decomposition by \eqref{eq:C} and Lemma \ref{lem:WLOG}. Because of this, as a parametrix for $C$, $T$ can be chosen with the same properties. 

To prove \eqref{eq:Nmpara} we follow the proof of \cite[Proposition 3]{chappa2006characterization} but with necessary changes to include the possibility of conjugate points. Let us first define the operator
\begin{equation}\label{eq:B}
B = \mathcal{N} + \D \Lambda \delta
\end{equation}
and let $S_D$ be the solution operator for the Dirichlet problem
\[
\delta d u = f, \quad
u|_{\partial M} = 0
\]
taking $f$ to $u$. Then applying $\D S_D$ on the right we have
\[
\D \Lambda = (B - \mathcal{N}) \D S_D
\]
and substituting this into \eqref{eq:B} using 
\eqref{eq:Ndecomp} yields
\[
\Psi + A + \D \Lambda \delta \equiv \mathcal{N} + (\Psi + A + \D \Lambda \delta-\mathcal{N}) \D S_D \delta.
\]
Next, note that $\mathcal{N} \D S_D \delta = I^* I \D S_D \delta = 0$ and so, using also \eqref{eq:C},
\[
C + A \equiv \mathcal{N} + (C + A) \D S_D \delta.
\]
Applying $T$ to the previous equation and rearranging yields
\[
T \mathcal{N} \equiv \mathrm{I} - d S_D\delta + TA  - T A d S_D \delta.
\]
The last term on the right side of the last equation is zero by Lemma \ref{lem:WLOG}, and since $\mathrm{I} - d S_D \delta = \pi_S$ we obtain the conclusion \eqref{eq:Nmpara}.
\end{proof}

Armed with Lemma \ref{lem:para}, we now present a result on cancellation of singularities which can occur in the presence of conjugate points. This is a version of \cite[Theorem 4.3]{monard2015geodesic} for the case of tensor fields. The statement uses microlocal Sobolev spaces which we first define as follows.

\begin{definition}\label{def:microS}
    Let $V \subset T^*M$ be an open conic set and $s \in \mathbb{R}$. Then $f \in \mathcal{D}'(M)$ is in $H^s(V)$ if for any function $\phi \in C^\infty(M)$ which is homogeneous of degree $1$ in the fiber variable and whose support is contained in $V$, $\phi(x,D)f \in H^s(M)$. The same definition applies for $f \in \mathcal{D}'(S^m\tau'_M)$, in which case we use the notation $H^s(V,S^m\tau'_M)$.
\end{definition}

We now give the cancellation of singularities result.

\begin{theorem} \label{thm:Canc}
Suppose that $\gamma_0$ is a geodesic in $M$ with  conjugate points $x_0$ and $y_0 \in M^{\rm int} \cap \gamma_0$. Let $V_i \subset T^*M^{\rm int}$, $i = 1$, $2$, be sufficiently small conical neighbourhoods around the wavefront elements in $\Lambda_J$, see \eqref{Lj}, corresponding to the conjugate pair and $\mathcal{V} \subset T^* \partial_+ SM$ a sufficiently small conical neighbourhood around $\dot{\gamma}_0 \cap \partial_+ SM$. For $i = 1$ or $2$, let $f_i \in \mathcal{D}'(S^m \tau'_{M})$ with $\mathrm{WF}(f_i) \subset V_i$. Then
\[
I_m(f_1+f_2) \in H^s(\mathcal{V})
\]
if and only if
\begin{equation}\label{eq:HsV1}
\pi_S f_1 + T A f_2 \in H^{s-1/2}(V_1,S^m \tau'_M).
\end{equation}
The previous statement is also true with $f_1$ and $f_2$ switched.
\end{theorem}

%\shc{Modify statement so that it doesn't mention $\wt{M}$ but assumes conjugate points are in interior. Statement changed, but we need to make sure this fits with the rest of the section.}

\begin{proof}
First, suppose that $V_1$ and $V_2$ are taken sufficiently small so that $C_I V_1$ and $C_I V_2 \subset \mathcal{V}$. Then we have
\begin{equation}\label{eq:WFI}
WF(I_m(f_1+f_2)) \subset C_I V_1 \cup C_I V_2 \subset \mathcal{V}.
\end{equation}
Now, suppose that
\[
I_m (f_1 +f_2) = h.
\]
Then
\[
T \mathcal{N} (f_1 + f_2) = TI^* h,
\]
and by \eqref{eq:Nmpara}
\begin{equation} \label{eq:CanId1}
\pi_S f_1 + T A f_2 + \pi_S f_2 + T A f_1 \equiv T I_m^* h.
\end{equation}
Now we will microlocalise \eqref{eq:CanId1}. Indeed, let us take a second pair of cones $W_i$ slightly larger that $V_i$ and let, for $i = 1$ or $2$, $
\phi_i \in C^\infty(T^*M)$ equal to $1$ on $W_i$ and with support in a conic neighbourhood of $W_i$. Then $f_i \equiv \phi_i(x,D) f_i$ and in the following we will abuse notation for the action of $\phi_i(x,D)$ on $f_i$ by simply writing $\phi_i f_i$. The conjugate pair $x_0$ and $y_0$ is associated with one component of the set \eqref{Lj} corresponding with operator $A_j$, and so we obtain
\begin{equation} \label{eq:CanId2}
\phi_1 \pi_S \phi_1 f_1 + \phi_1 T A_j \phi_2 f_2 \equiv \phi_1 T I_m^* h.
\end{equation}
Now suppose that $h \in H^s(\mathcal{V})$. Then since $I^*$ is an FIO of order $-1/2$ whose cannonical relation is a local graph and $T$ is a pseudodifferential operator of order $1$, $\phi_1 T I_m^* h \in H^{s-1/2}(V_1,S^m \tau'_M)$ which proves the ``only if'' part of the statement.
%\shc{Symbols need to be defined.}
Assume now \eqref{eq:HsV1} and that $\tilde{\phi}_1 \in C^\infty(T^*M)$ is a symbol of order $0$ equal to $1$ on $V_1$ with support contained in $W_1$. Then by \eqref{eq:HsV1} we have that the left-hand-side of \eqref{eq:CanId2} is in $H^{s-1/2}(V_1,S^m\tau'_M)$. If we apply $I_m \tilde{\phi}_1^2 C$ to \eqref{eq:CanId2}, by the mapping properties of $I_m$ and $C$ we find that the resulting left-hand-side will be in $H^{s+1}(\mathcal{V})$. Therefore, applying $I_m \tilde{\phi}_1^2 C$ to the right-hand-side of \eqref{eq:CanId2} we obtain
\[
H^{s+1}(\mathcal{V}) \ni I_m \tilde{\phi}_1^2 C \phi_1 T I_m^* h \equiv I_m \tilde{\phi}^2_1[\phi_1,C] T I_m^* h + I_m\tilde{\phi}_1^2 I_m^*h.
\]
Considering the supports of $\tilde{\phi}_1$ and $\phi_1$, we see that the first term on the right side above is smoothing and so
\begin{equation}\label{eq:CanId3}
H^{s+1}(\mathcal{V}) \ni I_m \tilde{\phi}_1^2 C \phi_1 T I_m^* h \equiv I_m \tilde{\phi}_1^2 I_m^*h \equiv (\tilde{\phi}_1I_m^*)^* (\tilde{\phi}_1I_m^*) h
\end{equation}
Suppose that $\psi \in C^\infty(T^*\partial_+ SM)$ is a symbol of order $0$ equal to one on $\mathcal{V}$ and supported in a slightly larger cone. By \eqref{eq:WFI}, $\psi(x,D) h \equiv h$ and \eqref{eq:CanId3} implies 
\[
\psi I_m \tilde{\phi}_1^2 I_m^* \psi h \in H^{s+1}(\mathcal{V}) = (\tilde{\phi}_1^*I_m^* \psi^*)^* (\tilde{\phi}_1I_m^* \psi) h.
\]
From the discussion after \eqref{eq:CI*}, $\tilde{\phi}_1 I_m^* \psi(x,D)$ satisfies the Bolker condition if $\mathcal{V}$ is sufficiently small and therefore $\mathcal{N}_{\tilde{\phi}_1^*I_m^*\psi^*} = (\tilde{\phi}_1^*I_m^*\psi^*)^* (\tilde{\phi}_1I_m^*\psi)$ is a pseudodifferential operator of order $-1$. Furthermore, since the canonical relation of $\tilde{\phi}_1^*I^*\psi^*$ is a graph (see \cite[Theorem 3.2]{monard2015geodesic}), we can apply the FIO composition calculus to determine the principal symbol of $\mathcal{N}_{\phi I_m^* \varphi}$ which, considering the definition \eqref{eq:I*}, will be a non-vanishing function multiplied by $|\xi|^{2m} = 1$  on $C_{I_m}V_1$ (note $\xi$ is defined through application of $C_{I_m^*}$ to points in this set). Thus the principal symbol of $\mathcal{N}_{\tilde{\phi}_1^*I^*\psi^*}$is non-vanshing on $C_{I_m}V_1$ and so this operator will be elliptic on that set. Thus there is a local parametrix for $(\tilde{\phi}_1I_m^*)^* (\tilde{\phi}_1I_m^*)$ on $C_IV_1$, and applying this to \eqref{eq:CanId3} completes the proof that $h \in H^s(\mathcal{V})$.
\end{proof}

Similar to \cite{monard2015geodesic} for the scalar case, we also have the following corollary showing that cancellation of singularities can actually occur.

\begin{corollary} \label{cor:Canc}
If there is a geodesic with a conjugate pair in $M$, then there exists solenoidal $f \in \mathcal{D}'(S^m \tau'_M)$ with nonempty wavefront set such that $I f \in C^\infty(\partial_+ SM)$. 
\end{corollary}
\begin{proof}
Let $V_1$ and $V_2$ be the same as in the statement of Theorem \ref{thm:Canc}. We will first show that there exists $f_2 \in \mathcal{D}'(S^m \tau'_M)$ which is solenoidal such that $\mathrm{WF}(f_2) \subset V_2$ is non-empty. Such a tensor field can be constructed by beginning with a scalar distribution $h$ with non-empty wavefront set contained in $V_2$ and support contained in the projection of $V_2$ onto $M$. Let us assume that $V_2$ is sufficiently small that this projection is contained in a coordinate patch, and that in these coordinates $V_2$ is a conic set around $\D\xi_1$. We continue the construction using these local coordinates to define a symmetric tensor field $\widetilde{f} \in \mathcal{D}'(S^m\tau_M')$ by
\begin{equation} \label{eq:wtf}
\widetilde{f}_{i_1 \cdot i_m} = \left \{
\begin{array}{cl}
    h, & \mbox{if $i_1 = \cdots = i_m = 2$.} \\
    0, & \mbox{otherwise.} 
\end{array} \right .
\end{equation}
Finally, let $f_2$ be the solenoidal part, as in Theorem \ref{thm:Solenoidal}, of $\widetilde{f}$. The map $\pi_s$ from symmetric tensor to its solenoidal component defined by Theorem \ref{thm:Solenoidal} is a pseudodifferential operator, and formula (2.6.3) from \cite[Lemma 2.6.1]{sharafutdinov1997integral} gives the principal symbol $\sigma_p(\pi_s)$. Using that formula with \eqref{eq:wtf} we have that
\[
\sigma_p(\pi_s)(\wt{f})(\D \xi_1)_{1 \cdots 1} = h.
\]
Therefore $WF(f_2)$ is not empty and contained in $V_2$.

Now, by Lemmas \ref{lem:WLOG} and \ref{lem:para}, we see that $TA_j f_2$ is solenoidal and so we can define $f_1 = - TA_j f_2$ which is solenoidal, and considering the canonical relation of $A_j$, has wavefront set contained in $V_2$. Considering the canonical relations of the other operators $A_i$ making up $A$ in \eqref{eq:Ndecomp1}, we also see that for $i \neq j$, $WF(A_i f_2)$ does not intersect $V_1$. Now set $f = f_1 + f_2$. Then, since $\pi_S f_1 + TA_j f_2 = 0 \in H^{s-1/2}(V_1,S^m\tau'_M)$ for all $s$, by Theorem \ref{thm:Canc} we see that $I_mf \in H^s(\mathcal{V})$ for all $s$. Since, considering the canonical relation of $I$, $\mathrm{WF}(I f) \subset \mathcal{V}$, this implies that $I_mf \in C^\infty(\partial_+ SM)$.
\end{proof}
\section*{Acknowledgments} 

Both authors would like to thank the Isaac Newton Institute for Mathematical Sciences, Cambridge, UK, for support and hospitality during the workshop, \emph{Rich and Nonlinear Tomography - a multidisciplinary approach} in 2023 where this work was initiated (supported by EPSRC Grant Number EP/R014604/1). Additionally, the second author acknowledges the support of the Department of Atomic Energy,  Government of India, under
Project No.  12-R\&D-TFR-5.01-0520. The first author was supported by the Engineering and Physical Sciences Research Council (EPSRC) through grant EP/V007742/1.
\bibliographystyle{plain}
\bibliography{Bibfile}
\end{document}